\documentclass[smallcondensed]{svjour3-id}

\usepackage{latexsym, amsfonts, amsmath, amssymb,verbatim,amscd,tikz}
\usepackage{authblk}

\renewcommand{\textwidth}{15cm}

\renewcommand{\topmargin}{-8mm}
\let\ds=\displaystyle

\newcommand{\BV}{{BV}}
\newcommand{\LG}{{LG}}

\newcommand{\CTPP}{\hbox{$CT\kern-0.2ex{P}\kern-0.2ex{P}$}}

\newcommand{\calA}{\mathcal{A}}
\newcommand{\calB}{\mathcal{B}}
\newcommand{\calP}{\mathcal{P}}
\newcommand{\calC}{\mathcal{C}}
\newcommand{\calM}{\mathcal{M}}

\newcommand{\mR}{\mathbb{R}}
\newcommand{\mC}{\mathbb{C}}

\newcommand{\mT}{\mathbb{T}}

\newcommand{\vecc}{{\boldsymbol{c}}}
\newcommand{\vecm}{{\boldsymbol{m}}}

\newcommand{\vecx}{{\boldsymbol{x}}}
\newcommand{\vecy}{{\boldsymbol{y}}}
\newcommand{\vecu}{{\boldsymbol{u}}}
\newcommand{\vecv}{{\boldsymbol{v}}}
\newcommand{\vecw}{{\boldsymbol{w}}}
\newcommand{\vecz}{{\boldsymbol{z}}}

\newcommand{\abs}[1]{\left\lvert#1\right\rvert}
\def\ls[#1,#2]{\overline{#1\,#2}}

\newcommand{\norm}[1]{\left\lVert#1\right\rVert}
\newcommand{\snorm}[1]{\bigl\lVert#1\bigr\rVert}

\newcommand{\normbv}[1]{\left\lVert#1\right\rVert_{\BV(\sigma)}}

\def\ipr<#1,#2>{\langle #1,#2\rangle}
\def\ls[#1,#2]{\overline{\vphantom{\vbox to 1.2 ex{}} #1\, #2}}

\newcommand{\st}{\,:\,}
\renewcommand{\Re}{\mathop{\mathrm{Re}}}

\newcommand{\PIC}{\mathrm{PIC}}
\newcommand{\UPIC}{\mathrm{UPIC}}

\DeclareMathOperator*{\var}{var}
\DeclareMathOperator*{\cvar}{\rm cvar}
\DeclareMathOperator*{\pvar}{\rm pvar}
\DeclareMathOperator{\vf}{vf}

\newcommand{\journalref}[1]{\textrm{#1}}

\begin{document}
\title{$AC(\sigma)$ spaces for polygonally inscribed curves}

\author{Shaymaa Al-shakarchi \and Ian Doust}

\institute{School of Mathematics and Statistics, University of New South Wales, UNSW Sydney 2052, Australia, \email{i.doust@unsw.edu.au }}


\maketitle

\keywords{Functions of bounded variation, $AC(\sigma)$ operators, isomorphisms of function spaces}
\subclass{Primary 46J10; Secondary 05C10,
46J45, 47B40, 26B30}
%
%
\begin{abstract}
For certain families of compact subsets of the plane, the isomorphism class of the algebra of absolutely continuous functions on a set is completely determined by the homeomorphism class of the set. This is analogous to the Gelfand--Kolmogorov theorem for $C(K)$ spaces. In this paper we define a family of compact sets comprising finite unions of convex curves and show that this family has the `Gelfand--Kolmogorov' property.
\end{abstract}

%
%

\section{Introduction}\label{intro}
$AC(\sigma)$ operators were introduced by Ashton and Doust \cite{AD4} as a generalization of normal operators to the Banach spaces setting. A core ingredient in the definition of these operators is the family of Banach algebras of absolutely continuous functions defined on nonempty compact subsets of the plane. These spaces were defined in terms of a new concept of two dimensional variation which was designed to ensure that the set of $AC(\sigma)$ operators had appropriate spectral properties.

 If $\sigma$ is a nonempty compact subset of $\mC$, we shall denote the algebra of absolutely continuous functions $f: \sigma \to \mC$ by $AC(\sigma)$. (Full definitions will be given in Section~\ref{S:Defs}.)
 There have now been a number of papers studying the structure of these spaces. The main question addressed has been when two such spaces are isomorphic as Banach algebras. Recall that the Gelfand--Kolmogorov Theorem \cite{GK,GJ} states that two spaces $C(K_1)$ and $C(K_2)$ are isomorphic as Banach algebras if and only if $K_1$ and $K_2$ are homeomorphic. For the $AC(\sigma)$ spaces, isomorphic function algebras must always have homeomorphic domain sets, but the converse implication may fail. Nonetheless, for certain families of compacts subsets of the plane, one does obtain a Gelfand--Kolmogorov type theorem.

\begin{definition} Let $\Sigma$ be a family of nonempty compact subsets of $\mC$. We shall say that $\Sigma$ is a \textbf{Gelfand--Kolmogorov family} if, for $\sigma,\tau \in \Sigma$, $AC(\sigma)$ is isomorphic to $AC(\tau)$ if and only if $\sigma$ is homeomorphic to $\tau$.
\end{definition}

The main result of \cite{DL} is that the family of all compact polygonal regions with finitely many polygonal holes is a Gelfand--Kolmogorov family. More recently it was shown that the family $LG$ of sets which are the union of finitely many closed line segments is also a Gelfand--Kolmogorov family \cite{ASD}. On the other hand, the family of countable compact subsets of the plane is not a Gelfand--Kolmogorov family \cite{DAS}. Nor indeed is the family of all compact sets which are homeomorphic to the unit interval \cite[Example 7.1]{ASD}.

The aim of this paper is to extend the result of \cite{ASD} to a significantly larger family of compact sets which we denote $\PIC$. Roughly speaking the class $\PIC$ consists of connected compact subsets of the plane which are finite unions of smooth convex curves. Typical $\PIC$ sets are shown in Figure~\ref{PIC-sets}. For technical reasons, we need to introduce some mild additional conditions on the curves.

\begin{figure}[ht!]
\begin{center}
\begin{tikzpicture}

\draw[ultra thick, black] (-4,-1) -- (-4,1);
\draw[ultra thick, black] (-5,0) -- (-3,0);
\draw (-4.5,0.7) node {$\sigma_1$};

\draw[ultra thick, black] (1.554, 0.) -- (1.554, 0.09776) -- (1.554, .1962) -- (1.553, .2963) -- (1.551, .3982) -- (1.546, .5024) -- (1.538, .6088) -- (1.523, .7166) -- (1.497, .8229) -- (1.458, .9253) -- (1.401, 1.018) -- (1.323, 1.094) -- (1.225, 1.150) -- (1.109, 1.181) -- (.9835, 1.189) -- (.8564, 1.179) -- (.7335, 1.156) -- (.6206, 1.129) -- (.5178, 1.100) -- (.4252, 1.074) -- (.3414, 1.051) -- (.2651, 1.033) -- (.1945, 1.018) -- (.1273, 1.008) -- (0.06304, 1.002) -- (0, 1.000) -- (-0.06304, 1.002) -- (-.1273, 1.008) -- (-.1941, 1.018) -- (-.2651, 1.033) -- (-.3414, 1.051) -- (-.4252, 1.074) -- (-.5178, 1.100) -- (-.6206, 1.129) -- (-.7335, 1.156) -- (-.8564, 1.179) -- (-.9835, 1.189) -- (-1.109, 1.181) -- (-1.225, 1.150) -- (-1.323, 1.095) -- (-1.401, 1.018) -- (-1.458, .9255) -- (-1.497, .8227) -- (-1.523, .7163) -- (-1.538, .6083) -- (-1.546, .5031) -- (-1.551, .3985) -- (-1.553, .2964) -- (-1.554, .1962) -- (-1.554, 0.09739) -- (-1.554, -0.0006329) -- (-1.554, -0.9711e-1) -- (-1.554, -.1959) -- (-1.553, -.2961) -- (-1.551, -.3983) -- (-1.546, -.5029) -- (-1.538, -.6095) -- (-1.523, -.7159) -- (-1.497, -.8224) -- (-1.458, -.9253) -- (-1.401, -1.018) -- (-1.323, -1.095) -- (-1.225, -1.151) -- (-1.109, -1.180) -- (-.9835, -1.189) -- (-.8564, -1.179) -- (-.7335, -1.156) -- (-.6206, -1.129) -- (-.5178, -1.100) -- (-.4252, -1.074) -- (-.3414, -1.051) -- (-.2651, -1.033) -- (-.1941, -1.018) -- (-.1273, -1.008) -- (-0.6304e-1, -1.002) -- (0, -1.000) -- (0.06304, -1.002) -- (.1273, -1.008) -- (.1941, -1.018) -- (.2651, -1.032) -- (.3414, -1.051) -- (.4252, -1.074) -- (.5178, -1.100) -- (.6206, -1.129) -- (.7335, -1.156) -- (.8564, -1.178) -- (.9835, -1.188) -- (1.109, -1.181) -- (1.225, -1.150) -- (1.323, -1.095) -- (1.401, -1.018) -- (1.458, -.9250) -- (1.497, -.8221) -- (1.523, -.7171) -- (1.538, -.6092) -- (1.546, -.5024) -- (1.551, -.3980) -- (1.553, -.2958) -- (1.554, -.1956) -- (1.554, -0.09831) -- (1.554, -0.0002880);
\draw (-1.1,0.8) node {$\sigma_2$};

\draw[ultra thick,black] (4.5,-0.2) circle (0.8cm);
\draw[ultra thick,black] (4.5,-1) -- (4.5,0.6);
\draw[ultra thick,black] (3.7,-0.2) -- (5.3,-0.2);
\draw[ultra thick, black] (4.5,0.6) arc(20:80:0.7);
\draw[ultra thick, black] (4.5,0.6) arc(160:100:0.7);
\draw (5.5,0.8) node {$\sigma_3$};

\end{tikzpicture}
\end{center}
\caption{Three polygonally inscribed curves}\label{PIC-sets}

\end{figure}
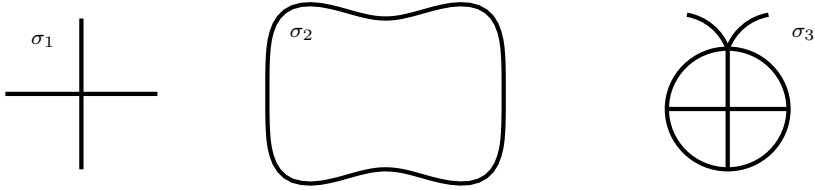

\begin{theorem}\label{MainThm} $\PIC$ is a Gelfand--Kolmogorov family.
\end{theorem}

Importantly, the unit circle $\mT$ is in $\PIC$. On reflexive Banach spaces $AC(\mT)$ operators are precisely the trigonometrically well-bounded operators introduced by Berkson and Gillespie \cite{BG}. Trigonometrically well-bounded operators have a well-developed structure theory, admitting an integral representation with respect to a suitable family of projections, and an extension of the functional calculus to all the functions of bounded variation on $\mT$. A consequence of the results of this paper, and the extension theorem proven in \cite{ASD3} is that if $T$ is an $AC(\sigma)$ operator on a reflexive Banach space and $\sigma \in \PIC$ is homeomorphic to the unit circle, then the $AC(\sigma)$ functional calculus for $T$ extends to a $BV(\sigma)$ functional calculus. It remains an open question as to whether this is true for all compact subsets $\sigma$ in the plane.

The proof of Theorem~\ref{MainThm} is structurally similar to the result for the family $\LG$. An important ingredient in both proofs is the ability to view $\PIC$ sets as the drawings of planar graphs. In Sections \ref{S:CC} to \ref{graph-isom} we introduce the family $\PIC$ and make connections to the parts of graph theory which we will use.

The main part of the proof of Theorem~\ref{MainThm} is to show that if $\sigma, \tau \in \PIC$ are homeomorphic, then  one can choose a homeomorphism $h: \sigma \to \tau$ such that $\Phi_h (f) = f \circ h^{-1}$ is an isomorphism of the corresponding spaces of absolutely continuous functions. The first main calculation is to show that $h$ can be chosen in a way which allows one to control the ratio of the variations of $f$ and $\Phi_h(f)$. This is done in Section~\ref{S:PIC-norm} by using the structure of the $\PIC$ sets to define a new norm which is equivalent to the original $BV$ norm.
 The final stage  in the proof (Sections \ref{S:AC-functions} to \ref{S:AC-join}) is to show that one may choose the algebra isomorphism between these spaces in such a way that it preserves the subalgebras of absolutely continuous functions, and consequently that $AC(\sigma)$ is isomorphic to $AC(\tau)$.
In the last section we shall make the minor extension to cover the case where the sets need not be connected.

Throughout, isomorphism will mean a Banach algebra isomorphism, that is, a continuous algebra isomorphism with a continuous inverse. (It is worth noting that for these particular algebras, the continuity is automatic; see \cite[Theorem 2.6]{DL}.) If Banach algebras $\calA$ and $\calB$ are isomorphic we shall write $\calA \simeq \calB$. We shall work throughout with algebras of complex-valued functions. We shall identify the plane as either $\mC$ or $\mR^2$ as is notationally convenient.

\section{Definitions}\label{S:Defs}

There are quite a number of concepts of variation for functions defined on subsets of the plane. The definitions used here were first introduced in \cite{AD}, but have undergone a number of simplifications (see \cite{DL}). Given that the motivation for these spaces comes from operator theory, it is important that the definitions apply to functions defined on the spectrum of a bounded operator, that is, on a nonempty compact subset $\sigma$ of the plane.

Suppose then that $f: \sigma \to \mC$.
Let $S = [\vecx_0,\vecx_1,\dots,\vecx_n]$ be a finite ordered list of elements of $\sigma$, where, for the moment, we shall assume that $n \ge 1$.
Let $\gamma_S$ denote the piecewise linear curve joining the points of $S$ in order.
Note that the elements of such a list do not need to be distinct. We will however require that consecutive points are different so that the line segment joining $\vecx_i$ to $\vecx_{i+1}$, denoted $\ls[\vecx_i,\vecx_{i+1}]$, is always a proper line segment.

The \textit{curve variation of $f$ on the set $S$} is defined to be
\begin{equation*} \label{lbl:298}
    \cvar(f, S) =  \sum_{i=1}^{n} \abs{f(\vecx_{i}) - f(\vecx_{i-1})}.
\end{equation*}
Associated to each list $S$ is its variation factor $\vf(S)$. Loosely speaking, this is the greatest number of times that $\gamma_S$ crosses any line in the plane.
To make this more precise we need the concept of a crossing segment.

\begin{definition}\label{crossing-defn}
Suppose that $\ell$ is a line in the plane. For $0 \le i < n$ we say that $\ls[\vecx_i,\vecx_{i+1}]$ is a \textbf{crossing segment} of $S = [\vecx_0,\vecx_1,\dots,\vecx_n]$ on $\ell$ if any one of the following holds:
\begin{enumerate}
  \item $\vecx_{i}$ and $\vecx_{i+1}$ lie on (strictly) opposite sides of $\ell$.
  \item $i=0$ and $\vecx_{i} \in \ell$.
  \item $\vecx_{i}\notin  \ell$ and $\vecx_{i+1} \in \ell$.

\end{enumerate}
\end{definition}

\begin{definition}\label{vf-defn}
Let $\vf(S,\ell)$ denote the number of crossing segments of $S$ on $\ell$. The \textbf{variation factor} of $S$ is defined to be
 $\ds \vf(S) = \max_{\ell} \vf(S,\ell)$.
\end{definition}

Clearly $1 \le \vf(S) \le n$. For completeness, in the case that
$S =[\vecx_0]$ we set $\cvar(f, [\vecx_0]) = 0$ and let $\vf([\vecx_0],\ell) = 1$ whenever $\vecx_0 \in \ell$.

\begin{definition}\label{2d-var}
The \textbf{two-dimensional variation} of a function $f : \sigma
\rightarrow \mathbb{C}$ is defined to be
\begin{equation*}
    \var(f, \sigma) = \sup_{S}
        \frac{ \cvar(f, S)}{\vf(S)},
\end{equation*}
where the supremum is taken over all finite ordered lists of elements of $\sigma$.
\end{definition}

The \textit{variation norm} of such a function is
  \[ \normbv{f} = \norm{f}_\infty + \var(f,\sigma) \]
and the set of functions of bounded variation on $\sigma$ is
  \[ \BV(\sigma) = \{ f: \sigma \to \mC \st \normbv{f} < \infty\}. \]
The space $\BV(\sigma)$ is a Banach algebra under pointwise operations \cite[Theorem 3.8]{AD}.

It is clear that if $\tau = \alpha \sigma + \beta$ is some nontrivial affine transformation of $\sigma$, then $BV(\sigma)$ is isometrically isomorphic to $BV(\tau)$. It is less obvious, but nonetheless true, that if  $\sigma = [a,b] \subseteq \mR$ then the above definition is equivalent to the more classical one. Importantly, $BV(\sigma)$ always contains all sufficiently smooth functions.

Let $\calP_2$ denote the algebra of complex polynomials in two real variables of the form $p(x,y) = \sum_{n,m} c_{nm} x^n y^m$, and let $\calP_2(\sigma)$ denote the restrictions of elements on $\calP_2$ to $\sigma$ (considered as a subset of $\mR^2$). The algebra $\calP_2(\sigma)$ is always a subalgebra of $\BV(\sigma)$ \cite[Corollary 3.14]{AD}.

\begin{definition}
The set of \textbf{absolutely continuous} functions on $\sigma$, denoted $AC(\sigma)$, is the closure of $\calP_2(\sigma)$ in $\BV(\sigma)$.
\end{definition}

The set $AC(\sigma)$ forms a closed subalgebra of $\BV(\sigma)$ and hence is a Banach algebra. Again, if $\sigma = [a,b]$ this definition reduces to the classical definition.

\section{Convex curves}\label{S:CC}

Let $C$ denote a finite length curve with parametrization $\gamma(t)$, $0 \le t \le L$ and endpoints $\vecx = \gamma(0)$ and $\vecy = \gamma(L)$. We shall say that $C$ is convex if
it has a supporting line through each point of the curve.
Convex curves are differentiable (that is, have a well-defined tangent) almost everywhere. To simplify matters we shall actually assume that each of our curves is differentiable, except at its endpoints. (Since we interested in sets which can be written as a union of such curves, this is not a major restriction.)
We shall generally use arc-length parameterizations, so that $\norm{\gamma'(t)} = 1$ for $0 < t < L$.

Let $\calC$ denote the set of differentiable convex curves in the plane with distinct endpoints.

\begin{definition} Suppose that $C \in \calC$ has endpoints $\vecx$ and $\vecy$.  We shall say that $C$ is \textbf{projectable} if  the orthogonal projection of $C$ onto the line through $\vecx$ and $\vecy$ is precisely the line segment $\ls[\vecx,\vecy]$. (See Figure~\ref{proj-cc-pic})
\end{definition}


\begin{figure}[ht!]
\begin{center}
\begin{tikzpicture}[scale=2]
\draw[ultra thick,blue, dashed] (0,0.3) -- (1,1.3);
\draw[ultra thick, black] (1,1.3) arc (90:180:1);
\draw (0,0.3) node[below] {$\vecx_1$};
\draw (1,1.3) node[right] {$\vecy_1$};
\draw (0,1) node {$C_1$};

\draw[ultra thick,black] (4.5970, .46028) -- (4.6239, .48117) -- (4.6484, .50253) -- (4.6703, .52433) -- (4.6896, .54652) -- (4.7064, .56904) -- (4.7204, .59186) -- (4.7318, .61492) -- (4.7404, .63816) -- (4.7463, .66154) -- (4.7494, .68500) -- (4.7498, .70850) -- (4.7474, .73198) -- (4.7423, .75538) -- (4.7344, .77867) -- (4.7238, .80178) -- (4.7105, .82467) -- (4.6945, .84728) -- (4.6759, .86957) -- (4.6547, .89148) -- (4.6309, .91297) -- (4.6047, .93399) -- (4.5760, .95449) -- (4.5449, .97443) -- (4.5115, .99376) -- (4.4759, 1.0124) -- (4.4381, 1.0304) -- (4.3983, 1.0477) -- (4.3564, 1.0642) -- (4.3127, 1.0799) -- (4.2672, 1.0948) -- (4.2200, 1.1087) -- (4.1712, 1.1218) -- (4.1209, 1.1340) -- (4.0692, 1.1452) -- (4.0163, 1.1554) -- (3.9622, 1.1646) -- (3.9072, 1.1727) -- (3.8512, 1.1799) -- (3.7944, 1.1859) -- (3.7370, 1.1909) -- (3.6791, 1.1948) -- (3.6208, 1.1977) -- (3.5622, 1.1994) -- (3.5035, 1.2000) -- (3.4448, 1.1995) -- (3.3862, 1.1979) -- (3.3278, 1.1952) -- (3.2698, 1.1914) -- (3.2124, 1.1866) -- (3.1555, 1.1806) -- (3.0994, 1.1736) -- (3.0442, 1.1656) -- (2.9901, 1.1565) -- (2.9370, 1.1464) -- (2.8852, 1.1353) -- (2.8347, 1.1233) -- (2.7858, 1.1103) -- (2.7384, 1.0965) -- (2.6926, 1.0817) -- (2.6487, 1.0661) -- (2.6066, 1.0497) -- (2.5665, 1.0325) -- (2.5285, 1.0146) -- (2.4926, .99602) -- (2.4590, .97676) -- (2.4276, .95690) -- (2.3986, .93646) -- (2.3720, .91550) -- (2.3480, .89406) -- (2.3265, .87220) -- (2.3075, .84995) -- (2.2912, .82738) -- (2.2776, .80452) -- (2.2667, .78143) -- (2.2585, .75816) -- (2.2530, .73477) -- (2.2503, .71130) -- (2.2504, .68780) -- (2.2532, .66433) -- (2.2588, .64094) -- (2.2671, .61768) -- (2.2781, .59460) -- (2.2918, .57174) -- (2.3082, .54918) -- (2.3272, .52695) -- (2.3489, .50510) -- (2.3730, .48368) -- (2.3997, .46274) -- (2.4288, .44233) -- (2.4602, .42248) -- (2.4940, .40324) -- (2.5299, .38467) -- (2.5680, .36678) -- (2.6082, .34964) -- (2.6503, .33326) -- (2.6944, .31770) -- (2.7402, .30298) -- (2.7876, .28914) -- (2.8367, .27621) -- (2.8872, .26421);
\draw[ultra thick,blue, dashed] (2.8872, .26421) -- (4.5970, .46028);
\draw[red,dashed] (2.8872, .26421) -- (2.4276, .95690) -- (4.5115, .99376) -- (4.5970, .46028);
\draw (2.8872, .26421)  node[below] {$\vecx_2$};
\draw (4.5970, .46028) node[below] {$\vecy_2$};
\draw (4.7494, .68500) node[right] {$C_2$};
\draw (2.35,1.05) node {$\vecv_1$};
\draw (4.6,1.1) node {$\vecv_2$};
\end{tikzpicture}
\end{center}
\caption{$C_1$ is a projectable convex curve, while $C_2$ is a nonprojectable convex curve, which could be split into three projectable curves.}\label{proj-cc-pic}
\end{figure}
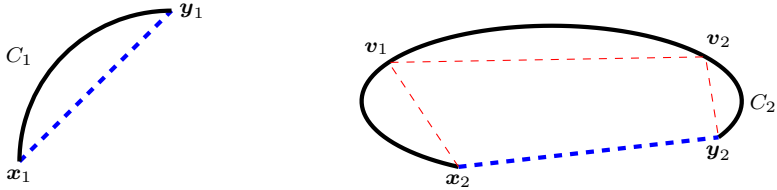

Given a nonprojectable convex curve $C$, one may always split it into  projectable curves. 
The proof of the following proposition is left to the reader.

\begin{proposition}\label{projectable}
Suppose that $C$ is a differentiable convex curve. Then $C$ can be split into a finite number of projectable curves.
\end{proposition}

\section{Polygonally inscribed curves}\label{S:PIC}

To be definite, the term polygon will mean a simple closed polygon including its interior, and so all polygons are homeomorphic to the closed disk.

\begin{definition} A (convex) \textbf{polygonal mosaic} in the plane is a finite collection $\calM$ of convex polygons such that
 \begin{enumerate}
  \item $\bigcup_{P \in \calM} P$ is connected.
  \item if $P \ne Q \in \calM$ intersect, then $P \cap Q$ is either
     \begin{itemize}
       \item  a single point which is a vertex of both $P$ and $Q$, or
       \item a line segment, which forms a full side of both $P$ and $Q$.
 \end{itemize}
 \end{enumerate}
\end{definition}

Some of our estimates will depend on the nature of the polygons which are elements of the mosaic.

\begin{definition}\label{S(M)-def}
For a polygon $P$ let $S(P)$ denote the number of sides of $P$. For a polygonal mosaic $\calM = \{P_i\}_{i=1}^M$, let $S(\calM) = \max_i \{S(P_i)\}$.
\end{definition}

\begin{definition} A nonempty compact connected set $\sigma$ is a \textbf{polygonally inscribed curve} if there exists a polygonal mosaic $\calM = \{P_i\}_{i=1}^M$ such that for each $i$, $\sigma \cap P_i$ is a differentiable convex curve $c_i$ joining two vertices of $P_i$ which only touches the boundary of $P_i$ at those points. The curves $c_i$ will be called the \textbf{components} of $\sigma$.

We shall denote the collection of all polygonally inscribed curves as $\PIC$.
\end{definition}

\begin{figure}[ht!]
\centering
\begin{tikzpicture}[xscale=1.6,yscale=1.2]

\draw[fill,green!10] (1,0) -- (1.5,1) -- (2,1) -- (2,0) -- (1.5,-0.5) -- (1,0);
\draw[fill,green!10] (1,1) -- (-0.866,1) -- (-0.866,0.5) -- (-1.5,0) -- (-0.866,-0.5) -- (-0.866,-1) -- (1,-1) -- (1,1);

\draw[blue,dashed] (1,0) -- (1.5,1) -- (2,1) -- (2,0) -- (1.5,-0.5) -- (1,0);
\draw[blue,dashed] (1,0) -- (0,0) -- (0,1) -- (1,1) -- (1,0);
\draw[blue,dashed] (0,1) -- (0,0) -- (-0.866,0.5) -- (-0.866,1) -- (0,1);
\draw[blue,dashed]  (-0.866,-0.5) -- (-1.5,0) -- (-0.866,0.5);
\draw[blue,dashed] (-0.866,-0.5) -- (0,0) -- (0,-1) -- (-0.866,-1) -- (-0.866,-0.5);
\draw[blue,dashed]  (0,-1) -- (1,-1) -- (1,0);

\draw[ultra thick,black] (1,0) parabola (2,1);
 \draw[ultra thick,black] (-0.866,-0.5) arc (-150:150:1);
 \draw[ultra thick,black] (-0.866,-0.5) arc (-30:30:1);

\node (x1) at (2,1) [circle,fill,red,scale=0.5] {};
\node (x2) at (1,0) [circle,fill,red,scale=0.5] {};
\node (x3) at (0,1) [circle,fill,red,scale=0.5] {};
\node (x4) at (-0.866,0.5) [circle,fill,red,scale=0.5] {};
\node (x5) at (-0.866,-0.5) [circle,draw,fill,red,scale=0.5] {};
\node (x6) at (0,-1) [circle,draw,fill,red,scale=0.5] {};

\draw (0.6,0.55) node {$\sigma$};

\end{tikzpicture}
\caption{A polygonally inscribed curve $\sigma$ with a suitable polygonal mosaic. The red dots represent the vertices of $G_{\sigma,\calM}$.}\label{poly-mos}
\end{figure}
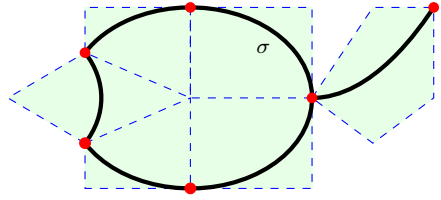

If we write below that $\sigma = \cup_{i=1}^M c_i$, then we will be implicitly assuming that there is a corresponding underlying polygonal mosaic $\calM$. As will be important later, the decomposition of a set $\sigma \in \PIC$ into a union of convex curves is far from unique.

It is worth noting that not every connected finite union of convex curves lies
in $\PIC$. Let $c_1 = \{(x, x^2) \st 0 \le x \le 1\}$, $c_2 = \{(x, x^3) \st 0 \le x \le 1\}$, and let
$\sigma = c_1 \cup c_2$. The curves $c_1$ and $c_2$ are smooth convex curves, but it is impossible to
find a polygonal mosaic ${P_1, P_2}$ so that $c_1 = \sigma \cap P_1$ and $c_2 = \sigma \cap  P_2$. Indeed, the
fact that $c_1$ and $c_2$ meet tangentially at $(0, 0)$ with the same convexity means that
one cannot get around this by splitting $\sigma$ into smaller pieces.

\section{$\PIC$ sets and graphs}\label{S:PIC-graphs}

Abstractly a graph $G = G(V,E)$ is determined by its vertex set $V$ and its edge set $E$. Given $\sigma = \cup_{i=1}^M c_i \in \PIC$, let $V_{\sigma,\calM}$ be the set of all endpoints of the curves $c_i$, and let $E_{\sigma,\calM} = \{c_i\}_{i=1}^M$ represent edges between points in $V_{\sigma,\calM}$. Thus $\sigma$ is a drawing of the graph $G_{\sigma,\calM} = G(V_{\sigma,\calM},E_{\sigma,\calM})$ (see Figure~\ref{poly-mos}).

Clearly the graph is not uniquely determined by $\sigma$. The definition of $\PIC$ precludes $G_{\sigma,\calM}$ from having any loops, but it may contain multiple edges between two vertices. It will be preferable later to avoid this situation so an important tool will be the Partition Lemma below which will ensure that if $\sigma \in \PIC$ we may always choose a partition $\calM$ so that $G_{\sigma,\calM}$ is a simple graph.

\begin{lemma}[Partition Lemma]\label{Part-Lem} Suppose that $P$ is a convex polygon and that
$c$ is a differentiable convex curve in $P$ joining one vertex $v_1$ of $P$ to another vertex $v_2$. Suppose that
$v$ is a point on $c$ in the interior of $P$. Then there exist convex polygons $P_1,  P_2  \subseteq P$
which only intersect at their boundaries and such that for $j = 1, 2$
\begin{enumerate}
 \item $v_j$ and $v$ are vertices of $P_j$ , and
 \item $c \cap P_j$ is a convex curve joining $v_j$ and $v$.
\end{enumerate}
\end{lemma}

In any specific example it is generally easy to do such a partitioning. Showing
that this is always possible requires the following general fact about convex curves and sets whose proof we leave to the reader.

\begin{lemma}\label{6.3.5} Suppose that $K$ is a closed convex subset of the plane with boundary $\partial K$ and that $\ell$ is
a line in the plane which intersects $\partial K$ in at least three places. Then
the line segment joining any two such points lies inside $\ell \cap \partial K$.
\end{lemma}

\begin{proof}[of Lemma~\ref{Part-Lem}]
 Suppose that P is a convex polygon and that $c$ is a
convex curve joining one vertex $\vecv_1$ of $P$ to another vertex $\vecv_2$. At each point $\vecw \in c$
there is a closed tangent half-plane which contains all of $c$. The intersection of
all these half-planes and the polygon $P$ is therefore a closed convex set $R_c$ whose
boundary consists of the curve $c$ and one or more of the sides of $P$.
Let $\vecm$ be any point on the boundary of $R_c$ which is not on $c$, and let $\ell_\vecm$ be the line through
$\vecv$ and $\vecm$. Note that by the lemma, $\ell_\vecm$ cannot intersect $\partial R_c$ at any
other points (since $\vecv$ does not lie on any line segment in the boundary of $R_c$ which
contains $\vecm$). It follows that $\ell_\vecm$ cuts $c$ into two parts at $\vecv$. Indeed if we intersect $P$
with the two closed half-planes bounded by $\ell_\vecm$, then we obtain two closed convex
polygons $Q_1$, chosen to contain the part of $c$ containing $\vecv_1$ and $\vecv$, and $Q_2$, containing
the part of $c$ containing $\vecv_2$ and $\vecv$.

Applying this construction to  a different point $\vecm' \in \partial R_\vecc \setminus c$ produces two polygons $Q_1'$ and $Q_2'$ with the same properties. Note that (again by the lemma), $\ell_{\vecm}$ and $\ell_{\vecm'}$ are distinct, and meet at $\vecv$. This means that if we set $P_1 = Q_1 \cap Q_1'$ and $P_2 = Q_2 \cap Q_2'$ then these polygons have a vertex at $\vecv$ and so satisfy the conclusions of the lemma.
\end{proof}

\begin{figure}\begin{center}
\begin{tikzpicture}[scale=1.75]
  \draw[fill,blue!10] (0,0) -- (0.555,-0.15) -- (1.2290,1.3190) -- (1.167,2) -- (0,2) -- (-0.7,1) -- (0,0);
  \draw[fill,green!10] (2,0) -- (2.7,0.7) -- (2,2) -- (1.53,2) -- (1.2290,1.3190) -- (1.4,-0.15) -- (2,0);
  \draw[thick,black] (0,0) -- (1,-0.2) -- (2,0) -- (2.7,0.7) -- (2,2) -- (0,2) -- (-0.7,1) -- (0,0);
  \draw[ultra thick,blue] (0,0) -- (1,-0.2) -- (2,0) -- (2.7,0.7);
  \draw[ultra thick,red] (2.7,0.7) -- (1.2290,1.3190) -- (1.1860,1.3270) -- (1.1430,1.3340) -- (1.1000,1.3390) -- (1.0560,1.3430) -- (1.0130,1.3440) -- (0.9695,1.3440) -- (0.9257,1.3410) -- (0.8831,1.3370) -- (0.8396,1.3320) -- (0.7965,1.3240) -- (0.7547,1.3150) -- (0.7123,1.3040) -- (0.6704,1.2910) -- (0.6291,1.2760) -- (0.5894,1.2600) -- (0.5495,1.2420) -- (0.5104,1.2220) -- (0.4730,1.2010) -- (0.4356,1.1780) -- (0.3992,1.1540) -- (0.3647,1.1280) -- (0.3305,1.1010) -- (0.2974,1.0720) -- (0.2664,1.0430) -- (0.2358,1.0110) -- (0.2067,0.9786) -- (0.1795,0.9455) -- (0.1531,0.9105) -- (0.1283,0.8744) -- (0.1054,0.8382) -- (0.0836,0.8002) -- (0.0634,0.7612) -- (0.0453,0.7224) -- (0.0283,0.6820) -- (0.0132,0.6409) -- (0.0001,0.6002) -- (-0.0120,0.5579) -- (-0.0220,0.5152) -- (-0.0300,0.4721) -- (-0.0360,0.4298) -- (-0.0400,0.3861) -- (-0.0430,0.3424) -- (-0.0440,0.2996) -- (-0.0430,0.2557) -- (-0.0400,0.2120) -- (-0.0360,0.1694) -- (-0.0290,0.1261) -- (-0.0210,0.0840) -- (-0.0110,0.0413) -- (0.0003,-0.0010);
  \draw[blue,dashed] (0.3, -0.6291) -- (1.7, 2.307);
  \draw[blue,dashed] (1.482, -0.6325) -- (1.12, 2.160);

  \draw[fill,black] (0,0) circle (1pt);
  \draw (-0.1,-0.1) node {$\vecv_1$};
  \draw[fill,black] (2.7,0.7) circle (1pt);
  \draw (2.7,0.7) node[right] {$\vecv_2$};
  \draw[fill,black] (1.2290,1.3190) circle (1pt);
  \draw (1.2290,1.360) node[right] {$\vecv$};
  \draw (0.6291,1.30) node[left] {$c$};
  \draw[fill,black] (0.54,-0.12) circle (1pt);
  \draw (0.7,-0.26) node {$\vecm$};
  \draw[fill,black] (1.42,-0.12) circle (1pt);
  \draw (1.7,-0.2) node {$\vecm'$};
  \draw (1.65,2.2) node[right] {$\ell_\vecm$};
  \draw (1.1,2.2) node[left] {$\ell_{\vecm'}$};
  \draw (0,1.5) node {$P_1$};
  \draw (2,1.5) node {$P_2$};
\end{tikzpicture}\end{center}
\caption{The construction in the proof of Lemma~\ref{Part-Lem}. The red and blue curves bound $R_c$.}\label{6.3.5 pic}
\end{figure}
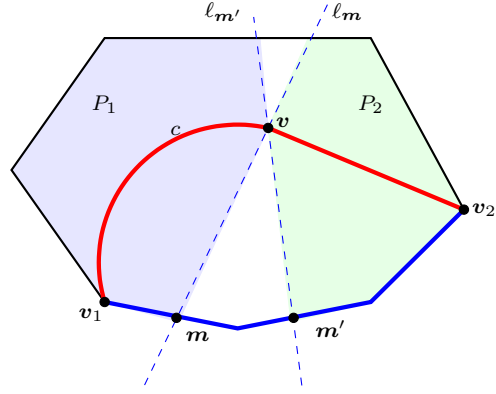

\begin{definition} We shall say that a polygonal mosaic  $\calM$ is a \textbf{simple polygonal mosaic for $\sigma$} if no two components
$c_1$ and $c_2$ of $\sigma$ have the same pair of endpoints. (That is, if $G_{\sigma,\calM}$ is a simple graph.)
\end{definition}

The following result is a consequence of Proposition~\ref{projectable} and the Partition Lemma.

\begin{theorem}\label{good-mosaic}
If $\sigma \in \PIC$ then there exists a simple polygonal mosaic $\calM = \{P_i\}_{i=1}^M$ for $\sigma$ such that each component curve $c_i = \sigma \cap P_i$ is a projectable convex curve.
\end{theorem}

\section{$\PIC$ sets and graph isomorphisms}\label{graph-isom}

An important component of the argument in \cite{ASD} is that there is a good correspondence between graph theoretic and the topological concepts of homeomorphism.

\begin{definition}
   Two simple graphs $ G_{1}(V_{1},E_{1}) $ and $ G_{2} (V_{2},E_{2}) $ are called \textbf{(graph) isomorphic} if there exists a bijective mapping, $ f : V_{1} \to V_{2} $ such that there is a edge between $v_{1} $ and $\hat{v}_{1}$ in $G_{1}$, if and only if there exists an edge between $ f(v_{1})$ and $f(\hat{v}_{1})$ in $G_{2} $.
\end{definition}


\begin{definition}
   A \textbf{subdivision} of an edge $\{u,v\}$ of a graph $G$ comprises forming a new graph with an additional vertex $w$, and replacing the edge $\{u,v\}$ with the two edges $\{u,w\}$ and $\{w,v\}$. A \textbf{subdivision} of $G$ is a graph formed by starting with $G$ and performing a finite sequence of subdivisions of edges.
  \end{definition}

If $G = G_{\sigma,\calM}$ is the graph associated to a set $\sigma  = \cup_{i=1}^M c_i \in \PIC$  with associated polygonal mosaic $\calM$, then subdividing an edge just corresponds to splitting one of the curves $c_i$ into two parts. By the Partition Lemma, this new representation of $\sigma$ has a corresponding polygonal mosaic $\calM'$. Of course if $\calM$ is a simple polygonal mosaic for $\sigma$ and all the curves $c_i$ are projectable, then the new representation also has these properties. A subdivision of $G_{\sigma,\calM}$ corresponds to a finite sequence of such curve splittings.

Note in particular that if $\sigma \in \PIC$ has two decompositions into component curves corresponding to two polygonal mosaics $\calM_1$ and $\calM_2$, then, taking the vertex set $V = V_{\sigma,\calM_1} \cup V_{\sigma,\calM_2}$, one can find a common subdivision of the two associated graphs. Each curve in either of the original decompositions of $\sigma$ is then a finite `concatenation' of curves in the subdivision decomposition.

\begin{definition}
    Two graphs  $ G_{1} $ and  $ G_{2} $ are \textbf{graph homeomorphic} if there is a graph isomorphism from some subdivision of $ G_{1} $ to some subdivision of $ G_{2} $.
\end{definition}

For the graphs which concern us, the two notions of homeomorphism agree. A proof of the direction we need here is given in \cite{ASD}.

\begin{theorem}\cite[p.~18]{GT}\label{graph-topol hom}
  Suppose that $G_1$ and $G_2$ are planar graphs with drawings ${\hat G}_1$ and ${\hat G}_2$ in the plane. Then
  $G_1$ and $G_2$ are graph homeomorphic if and only if ${\hat G}_1$ and ${\hat G}_2$ are topologically homeomorphic.
\end{theorem}

Suppose that $\sigma, \tau \in \PIC$ are homeomorphic subsets of the plane. By Theorem~\ref{good-mosaic} there are simple polygonal mosaics $\calM_\sigma$ and $\calM_\tau$ for each of these sets and hence the corresponding graphs $G_\sigma = G_{\sigma,\calM_\sigma}$ and $G_\tau = G_{\tau,\calM_\tau}$ are simple graphs with drawings $\sigma$ and $\tau$.

By Theorem~\ref{graph-topol hom}, $G_\sigma$ and $G_\tau$ are homeomorphic graphs, and hence they admit subdivisions ${\hat G}_\sigma$ and ${\hat G}_\tau$ which are isomorphic graphs. Let $H: {\hat G}_\sigma \to {\hat G}_\tau$ denote the graph isomorphism.
By repeatedly applying the Partition Lemma, we can produce new simple polygonal mosaics ${\widehat \calM}_\sigma = \{P_i\}_{i=1}^M$ and ${\widehat \calM}_\tau = \{P_i'\}_{i=1}^M$ ordered in such a way that for all $i$, the edge $c'_i = P'_i \cap \tau$ in ${\hat G}_\tau$ is the image under $H$ of the edge $c_i = P_i \cap \sigma$ in ${\hat G}_\sigma$.

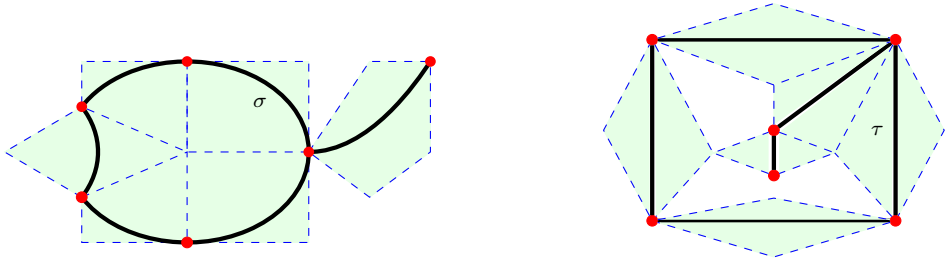
\begin{figure}[ht!] \begin{center}
\begin{tikzpicture}[xscale=1.6,yscale=1.2]

\draw[fill,green!10] (1,0) -- (1.5,1) -- (2,1) -- (2,0) -- (1.5,-0.5) -- (1,0);
\draw[fill,green!10] (1,1) -- (-0.866,1) -- (-0.866,0.5) -- (-1.5,0) -- (-0.866,-0.5) -- (-0.866,-1) -- (1,-1) -- (1,1);

\draw[blue,dashed] (1,0) -- (1.5,1) -- (2,1) -- (2,0) -- (1.5,-0.5) -- (1,0);
\draw[blue,dashed] (1,0) -- (0,0) -- (0,1) -- (1,1) -- (1,0);
\draw[blue,dashed] (0,1) -- (0,0) -- (-0.866,0.5) -- (-0.866,1) -- (0,1);
\draw[blue,dashed]  (-0.866,-0.5) -- (-1.5,0) -- (-0.866,0.5);
\draw[blue,dashed] (-0.866,-0.5) -- (0,0) -- (0,-1) -- (-0.866,-1) -- (-0.866,-0.5);
\draw[blue,dashed]  (0,-1) -- (1,-1) -- (1,0);

\draw[ultra thick,black] (1,0) parabola (2,1);
 \draw[ultra thick,black] (-0.866,-0.5) arc (-150:150:1);
 \draw[ultra thick,black] (-0.866,-0.5) arc (-30:30:1);

\node (x1) at (2,1) [circle,fill,red,scale=0.5] {};
\node (x2) at (1,0) [circle,fill,red,scale=0.5] {};
\node (x3) at (0,1) [circle,fill,red,scale=0.5] {};
\node (x4) at (-0.866,0.5) [circle,fill,red,scale=0.5] {};
\node (x5) at (-0.866,-0.5) [circle,draw,fill,red,scale=0.5] {};
\node (x6) at (0,-1) [circle,draw,fill,red,scale=0.5] {};

\draw (0.6,0.55) node {$\sigma$};

\end{tikzpicture}
\hspace{2cm}
\begin{tikzpicture}[xscale=1.6,yscale=1.2]
\path
node (y1) at (0,-0.5) [circle,draw,fill,red,scale=0.5] {}
node (y2) at (0,0) [circle,draw,fill,red,scale=0.5] {}
node (y3) at (1,1) [circle,draw,fill,red,scale=0.5] {}
node (y4) at (-1,1) [circle,draw,fill,red,scale=0.5] {}
node (y5) at (-1,-1) [circle,draw,fill,red,scale=0.5] {}
node (y6) at (1,-1) [circle,draw,fill,red,scale=0.5] {} ;

\draw[ultra thick,black] (y1) -- (y2) -- (y3) -- (y4) -- (y5) -- (y6) -- (y3);

\draw[fill,green!10] (y1) -- (-0.5,-0.25) -- (y2) -- (0.5,-0.25)  -- (y1);
\draw[fill,green!10] (y2) -- (0,0.5) -- (y3) -- (0.5,-0.25) -- (y2);
\draw[fill,green!10] (y3) -- (0,0.5) -- (y4) -- (0,1.4) -- (y3);
\draw[fill,green!10] (y4) -- (-0.5,-0.25) -- (y5) -- (-1.4,0) -- (y4);
\draw[fill,green!10] (y5) -- (0,-0.75) -- (y6) -- (0,-1.4) -- (y5);
\draw[fill,green!10] (y3) -- (1.4,0) -- (y6) -- (0.5,-0.25) -- (y3);

\draw[blue,dashed] (y1) -- (-0.5,-0.25) -- (y2) -- (0.5,-0.25)  -- (y1);
\draw[blue,dashed] (y2) -- (0,0.5) -- (y3) -- (0.5,-0.25);
\draw[blue,dashed]  (0,0.5) -- (y4) -- (0,1.4) -- (y3);
\draw[blue,dashed]  (y4) -- (-0.5,-0.25) -- (y5) -- (-1.4,0) -- (y4);
\draw[blue,dashed] (y5) -- (0,-0.75) -- (y6) -- (0,-1.4) -- (y5);
\draw[blue,dashed]  (y3) -- (1.4,0) -- (y6) -- (0.5,-0.25);

\draw (0.85,0) node {$\tau$};

\end{tikzpicture}\end{center}
\caption{Two homeomorphic sets $\sigma,\tau \in \PIC$ with  simple polygonal mosaics.}
\end{figure}

\begin{figure}[ht!]
\begin{center}
\begin{tikzpicture}[xscale=1.6,yscale=1.2]

\draw[fill,green!10]  (1,0) -- (1.25,0.5) -- (1.5,0.25) --  (1.25,-0.25) -- (1,0);
\draw[fill,green!10] (1.5,0.25) -- (1.5,1) -- (2,1) --  (2,0) -- (1.5,0.25);

\draw[fill,green!10] (1,1) -- (-0.866,1) -- (-0.866,0.5) -- (-1.5,0) -- (-0.866,-0.5) -- (-0.866,-1) -- (1,-1) -- (1,1);

\draw[ultra thick,black] (1,0) parabola (2,1);
 \draw[ultra thick,black] (-0.866,-0.5) arc (-150:150:1);
 \draw[ultra thick,black] (-0.866,-0.5) arc (-30:30:1);

\draw (1.75,0.8) node {\tiny{$1$}};
\draw (1.25,0.2) node {\tiny{$2$}};
\draw (0.8,0.8) node {\tiny{$3$}};
\draw (-0.5,0.7) node {\tiny{$4$}};
\draw (-0.85,0) node {\tiny{$5$}};
\draw (-0.5,-0.7) node {\tiny{$6$}};
\draw (0.8,-0.8) node {\tiny{$7$}};

\draw[blue,dashed] (1,0) -- (1.25,0.5) -- (1.5,0.25) --  (1.25,-0.25) -- (1,0);
\draw[blue,dashed] (1.5,0.25) -- (1.5,1) -- (2,1) --  (2,0) -- (1.5,0.25);
\draw[blue,dashed] (1,0) -- (0,0) -- (0,1) -- (1,1) -- (1,0);
\draw[blue,dashed] (0,1) -- (0,0) -- (-0.866,0.5) -- (-0.866,1) -- (0,1);
\draw[blue,dashed]  (-0.866,-0.5) -- (-1.5,0) -- (-0.866,0.5);
\draw[blue,dashed] (-0.866,-0.5) -- (0,0) -- (0,-1) -- (-0.866,-1) -- (-0.866,-0.5);
\draw[blue,dashed]  (0,-1) -- (1,-1) -- (1,0);

\node (x1) at (2,1) [circle,fill,red,scale=0.5] {};
\node (x2) at (1,0) [circle,fill,red,scale=0.5] {};
\node (x3) at (0,1) [circle,fill,red,scale=0.5] {};
\node (x4) at (-0.866,0.5) [circle,fill,red,scale=0.5] {};
\node (x5) at (-0.866,-0.5) [circle,draw,fill,red,scale=0.5] {};
\node (x6) at (0,-1) [circle,draw,fill,red,scale=0.5] {};
\node (x7) at (1.5,0.25) [circle,draw,fill,blue,scale=0.6] {};

\end{tikzpicture}
\hspace{2cm}
\begin{tikzpicture}[xscale=1.6,yscale=1.2]
\path
node (y1) at (0,-0.5) [circle,draw,fill,red,scale=0.5] {}
node (y2) at (0,0) [circle,draw,fill,red,scale=0.5] {}
node (y3) at (1,1) [circle,draw,fill,red,scale=0.5] {}
node (y4) at (-1,1) [circle,draw,fill,red,scale=0.5] {}
node (y5) at (-1,-1) [circle,draw,fill,red,scale=0.5] {}
node (y6) at (1,-1) [circle,draw,fill,red,scale=0.5] {}
node (y7) at (-1,0) [circle,draw,fill,blue,scale=0.6] {};

\draw[ultra thick,black] (y1) -- (y2) -- (y3) -- (y4) -- (y5) -- (y6) -- (y3);


\draw[fill,green!10] (y1) -- (-0.5,-0.25) -- (y2) -- (0.5,-0.25)  -- (y1);
\draw[fill,green!10] (y2) -- (0,0.5) -- (y3) -- (0.5,-0.25) -- (y2);
\draw[fill,green!10] (y3) -- (0,0.5) -- (y4) -- (0,1.4) -- (y3);
\draw[fill,green!10] (y4) -- (-0.75,0.375) -- (y7) -- (-1.2,0.5) -- (y4);
\draw[fill,green!10] (y7) -- (-0.75,-0.375) -- (y5) -- (-1.2,-0.5) -- (y7);
\draw[fill,green!10] (y5) -- (0,-0.75) -- (y6) -- (0,-1.4) -- (y5);
\draw[fill,green!10] (y3) -- (1.4,0) -- (y6) -- (0.5,-0.25) -- (y3);

\draw[blue,dashed] (y1) -- (-0.5,-0.25) -- (y2) -- (0.5,-0.25)  -- (y1);
\draw[blue,dashed] (y2) -- (0,0.5) -- (y3) -- (0.5,-0.25);
\draw[blue,dashed]  (0,0.5) -- (y4) -- (0,1.4) -- (y3);
\draw[blue,dashed]  (y4) -- (-0.75,0.375) -- (y7) -- (-1.2,0.5) -- (y4);
\draw[blue,dashed]  (y7) -- (-0.75,-0.375) -- (y5) -- (-1.2,-0.5) -- (y7);
\draw[blue,dashed] (y5) -- (0,-0.75) -- (y6) -- (0,-1.4) -- (y5);
\draw[blue,dashed]  (y3) -- (1.4,0) -- (y6) -- (0.5,-0.25);

\draw (0.1,-0.25) node {\tiny{$1$}};
\draw (0.45,0.25) node {\tiny{$2$}};
\draw (0,1.15) node {\tiny{$3$}};
\draw (-0.9,0.5) node {\tiny{$4$}};
\draw (-0.9,-0.5) node {\tiny{$5$}};
\draw (0,-1.2) node {\tiny{$6$}};
\draw (1.1,0) node {\tiny{$7$}};
\end{tikzpicture}\end{center}
\caption{The refined mosaics ${\widehat \calM}_\sigma$ and ${\widehat \calM}_\tau$ with the corresponding curves labelled so that the graphs ${\hat G}_\sigma$ and ${\hat G}_\tau$ are graph isomorphic.}
\end{figure}
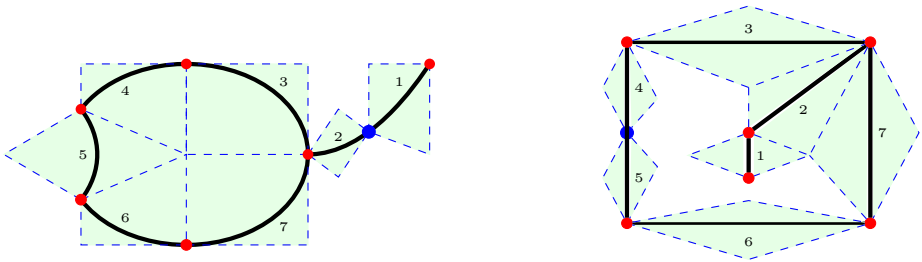

The arc-length parameterizations of the curves generate natural homeomorphisms $h_i: c_i \to c_i'$ for $i = 1,\dots,M$. If the directions of these parameterizations are chosen appropriately we can ensure that if $\vecx$ is an endpoint of both $c_i$ and $c_j$ then $h_i(\vecx) =   h_j(\vecx)$, and hence there exists a homeomorphism $h: \sigma \to \tau$ determined by $h(\vecz) = h_i(\vecz)$ for $z \in c_i$.

\section{The $\PIC$ norm}\label{S:PIC-norm}

Given any parametrized curve, there is a natural sense of variation of a function alone the curve, and this is usually easier to compute than our two dimensional variation. The main result in this section is that if $\sigma \in \PIC$ one can always define a norm using this `parametrized variation' which is equivalent to the $BV$ norm which we introduced earlier.

Let $C$ denote a finite length curve with parametrization $\gamma(t)$, $0 \le t \le L$ and endpoints $\vecx = \gamma(0)$ and $\vecy = \gamma(L)$.
Suppose that $f: C \to \mC$. The parameterized variation of $f$ on $C$ is
  \[ \pvar(f,C) = \var_{[0,L]} (f \circ \gamma) = \sup \sum_{i=1}^n |f(\gamma(t_i)) - f(\gamma(t_{i-1}))| \]
where the supremum is taken over all finite partitions $0 \le t_0 < t_1 < \dots < t_n \le L$ of the parameter set. Note that $\pvar(f,C)$ does not depend on the actual parameterization --- any continuous one-to-one function $\gamma$ mapping an interval to $C$ will do. For later we record  the following easy fact.

\begin{lemma}\label{pvar-join} Let $C$ be a finite length curve with parameterization $\gamma(t)$, $0 \le t \le L$. Suppose that $0 < L_0 < L$, that $C_1$ is the part of $C$ from $\gamma(0)$ to $\gamma(L_0)$ and that $C_2$ is the part of $C$ from $\gamma(L_0)$ to $\gamma(L)$. Then for any $f: C \to \mC$,
  \[ \pvar(f,C) = \pvar(f,C_1) + \pvar(f,C_2). \]
\end{lemma}

\begin{theorem}\label{equiv-C} Suppose that $C$ is a projectable convex finite length curve. Then for any $f: C \to \mC$,
    \[ \var(f,C) \le \pvar(f,C) \le 2 \var(f,C). \]
\end{theorem}

\begin{proof}
To simplify the arguments, we shall take advantage of the affine invariance of these quantities. By taking an appropriate affine transformation of the plane we can assume that $C$ is of the form $\{\gamma(t) = (t,g(t)) \st 0 \le t \le 1\}$ for some continuous function $g: [0,1] \to \mR$.

Consider any ordered partition $\{t_i\}_{i=0}^n$ of the parameter set $[0,1]$ and let $S = [\vecx_0,\vecx_1,\dots,\vecx_n]$ where $\vecx_i = \gamma(t_i)$.
Since $C$ is convex, $1 \le \vf(S) \le 2$ and hence
   \[
   \sum_{i=1}^n |f(\gamma(t_i)) - f(\gamma(t_{i-1}))|
    \le 2\, \frac{\cvar(f,S)}{\vf(S)} \\
   \le 2 \var(f,C)
   \]
which proves the right hand inequality.

Suppose now that $S = [\vecx_0,\vecx_1,\dots,\vecx_n]$ is an arbitrary list of elements of $C$. The set of distinct points in the list $S$ can be labelled $\gamma(t_0),\gamma(t_1),\dots,\gamma(t_m)$ for some ordered set of parameters $t_0 < t_1 < \dots < t_m$ with $m \le n$.
 For $i = 1,2,\dots,m$ let $I_i = [t_{i-1},t_i]$.

 Our aim is to bound $\cvar(f,S)/\vf(S)$.
Consider the  term $|f(\vecx_j) - f(\vecx_{j-1})|$, where $\vecx_{j-1} = \gamma(t_{i_1})$ and $\vecx_j = \gamma(t_{i_2})$. Suppose first that $i_i < i_2$. Then by the triangle inequality
  \[ |f(\vecx_j) - f(\vecx_{j-1})|
  \le \sum_{i=i_1+1}^{i_2} |f(\gamma(t_i)) - f(\gamma(t_{i-1}))|
  \le \sum_{i=i_1+1}^{i_2} \var_{I_i} (f \circ \gamma) . \]
A similar argument applies if $i_1 > i_2$. Adding all the terms gives that
  \begin{equation}\label{sum-intervals}
  \sum_{j=1}^n |f(\vecx_j) - f(\vecx_{j-1})|
  \le \sum_{i=1}^m k_i \, \var_{I_i} (f \circ \gamma)
  \end{equation}
where $k_i$ is the number of values of $j$ that the interval $I_i$ lies between $\vecx_{j-1}$ and $\vecx_j$
(so $1 \le k_i \le m$). Let $k = \max\{k_1,k_2,\dots, k_m\}$ and choose a value $i_0$ so that $k = k_{i_0}$. (See Figure~\ref{c-curve-arg})

\begin{figure}[ht!]
\centering
 \begin{tikzpicture}[xscale=4,yscale=7]

 \draw[red,<->] (-0.2,0) -- (1.2,0);
 \draw[green,thick] (0.52,0) -- (0.52,0.37);
 \draw (0.52,0.34) node[right] {$\ell$};

 \draw[red] (0.1,-0.01) -- (0.1,0.01);

 \draw[red] (0.4,-0.01) -- (0.4,0.01);
 \draw[red] (0.25,0) node[below] {$I_1$};
 \draw[red] (0.7,-0.01) -- (0.7,0.01);
 \draw[red] (0.55,0) node[below] {$I_2$};
 \draw[red] (0.9,-0.01) -- (0.9,0.01);
 \draw[red] (0.8,0) node[below] {$I_3$};

 \draw[red] (0.1,0) node[above] {\tiny{$t_0$}};
 \draw[red] (0.4,0) node[above] {\tiny{$t_1$}};
 \draw[red] (0.7,0) node[above] {\tiny{$t_2$}};
 \draw[red] (0.9,0) node[above] {\tiny{$t_3$}};

  \draw[black, line width = 0.50mm]   plot[smooth,domain=0:1] (\x, {\x-\x^2});

  \draw[blue] (0.1,0.09) -- (0.9,0.09) -- (0.4,0.24) -- (0.1,0.09) -- (0.7,0.21) -- (0.4,0.24);
  \draw[blue,->]  (0.1,0.09) -- (0.5,0.09);
  \draw[blue,->] (0.9,0.09) -- (0.7,0.15);
  \draw[blue,->] (0.4,0.24) -- (0.25,0.165);
  \draw[blue,->] (0.1,0.09) -- (0.3,0.13);
  \draw[blue,->] (0.7,0.21) -- (0.55,0.225);

  \draw (0.1,0.09) node[left] {$\vecx_0 = \vecx_3$};
  \draw (0.1,0.09) node[circle, draw, fill=black,inner sep=0pt, minimum width=4pt] {};
   \draw (0.4,0.27) node[left] {$\vecx_2 = \vecx_5$};
   \draw (0.4,0.24) node[circle, draw, fill=black,inner sep=0pt, minimum width=4pt] {};
   \draw (0.7,0.22) node[right] {$\vecx_4$};
   \draw (0.7,0.21) node[circle, draw, fill=black,inner sep=0pt, minimum width=4pt] {};
   \draw (0.9,0.09) node[right] {$\vecx_1$};
   \draw (0.9,0.09) node[circle, draw, fill=black,inner sep=0pt, minimum width=4pt] {};

 \end{tikzpicture}
\caption{Example of the quantities in the proof of Theorem~\ref{equiv-C}. In this example $k = 3$ and $i_0 = 2$.}\label{c-curve-arg}
\end{figure}
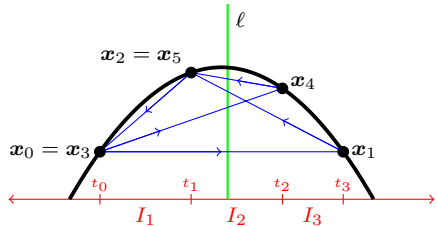

If $t$ is any number in the interior of $I_{i_0}$ and $\ell$ is the vertical line through $t$, then $S$ necessarily has $k$ crossing segments on $\ell$, and hence $\vf(S) \ge k$.

But (\ref{sum-intervals}), and the additivity of variation over contiguous intervals, implies that
  \begin{equation}\label{cvar-ub} \cvar(f,S)
  \le k \sum_{i=1}^m \var_{I_i} (f \circ \gamma)
  = k \var_{[t_0,t_m]} (f \circ \gamma)
  \le k \var_{[0,1]} (f \circ \gamma)
   = k \pvar(f,C).
  \end{equation}
Combining these shows that
   \[ \frac{\cvar(f,S)}{\vf(S)} \le \frac{k \pvar(f,C)}{k} = \pvar(f,C)\]
and hence $\var(f,C) \le \pvar(f,C)$.
\end{proof}

\begin{definition}\label{PIC-norm} Suppose that $\sigma = \cup_{i=1}^M c_i \in \PIC$ with simple polygonal mosaic $\calM$ such that $c_i$ is projectable for each $i$. For $f: \sigma \to \mC$ let
  \[ \norm{f}_{\PIC(\sigma)} = \norm{f}_\infty + \sum_{i=1}^M \pvar(f,c_i). \]
We will denote the set of all functions $f$ such that $\norm{f}_{\PIC(\sigma)}< \infty$ as $\PIC(\sigma)$.
\end{definition}

At first glance it may appear that we should denote this quantity as $\norm{f}_{\PIC(\sigma,\calM)}$ since it depends on the decomposition of $\sigma$ into convex curves. As was noted in Section~\ref{graph-isom} however, given any two decompositions of $\sigma$ with associated polygonal mosaics $\calM_1$ and $\calM_2$, there is a common finer decomposition into smaller curves with mosaic $\calM_{12}$. We may now apply Lemma~\ref{pvar-join} to see that $\norm{f}_{\PIC(\sigma),\calM_1} = \norm{f}_{\PIC(\sigma),\calM_{12}}  = \norm{f}_{\PIC(\sigma),\calM_2} $. Consequently we may safely omit mention of the mosaic and write $\norm{f}_{\PIC(\sigma)}$.

Since $\norm{\cdot}_\infty$ and each of the terms $\pvar(\cdot,c_i)$ have the appropriate homogeneity and subadditivity properties it is easy to verify that $\norm{\cdot}_{\PIC(\sigma)}$ is a norm. Our first aim is to show that it is equivalent to the $BV(\sigma)$ norm.

Note that by Theorem~\ref{equiv-C} and the fact that $\var(f,c_i) \le \var(f,\sigma)$ for all $i$,
  \[
   \norm{f}_{\PIC(\sigma)}
     \le \norm{f}_\infty + \sum_{i=1}^M 2 \var(f,c_i)
     \le \norm{f}_\infty + 2M \var(f,\sigma)
     \le 2M \norm{f}_{BV(\sigma)}.
  \]
Proving an inequality in the reverse direction is more difficult, and indeed the constants involved depend on geometric properties of the polygonal mosaic.

\begin{lemma} Suppose that $\sigma \in \PIC$. Then there exists a constant $K_\sigma$ such that
    \[ \norm{f}_{BV(\sigma)}
   \leq K_\sigma \norm{f}_{\PIC(\sigma)} \]
for all $f\in \PIC(\sigma).$
\end{lemma}

\begin{proof} Suppose that $\sigma = \cup_{i=1}^M c_i$ is a decomposition of $\sigma$ into projectable convex curves coming from a simple polygonal mosaic $\calM$.
For $k=1,2,\dots,M$, let $\sigma_{k}=\cup_{j=1}^{k}c_{j}$. We shall assume the curves $c_i$ have been ordered so that each set $\sigma_k$ is connected.

Suppose that $f: \sigma \to \mC$. We shall proceed by induction  to show that
  \begin{equation}\label{bv-pic-induction}
    \norm{f}_{\BV(\sigma_{m})}
    \leq \bigl(m+2(m-1)S(\calM)\bigr)  \norm{f}_{\PIC(\sigma_{m})}
  \end{equation}
for $m$ from $1$ to $M$. (Recall that $S(\calM)$ was defined in Definition~\ref{S(M)-def}.)

Since $\sigma_1 = c_1$ is a projectable convex curve, it follows from Theorem~\ref{equiv-C} that
 \[ \norm{f}_{\BV(\sigma_1)}
  = \sup_{\vecz \in \sigma_1} | f(\vecz)| + \var(f,\sigma_1)
  \le \norm{f}_{\PIC(\sigma_1)} \]
so (\ref{bv-pic-induction}) holds when $m = 1$.

Suppose now that $1\leq k < M$, and that (\ref{bv-pic-induction}) holds if $m = k$. Let  $S = [\vecz_0,\vecz_1,\dots,\vecz_n]$  be a list of points in $\sigma_{k+1} = \sigma_k \cup c_{k+1} $. Denote the endpoints of $c_{k+1} $ by $\vecx$ and $\vecy$. For the moment assume that both $\vecx$ and $\vecy$ are elements of $\sigma_k$.

\begin{figure}[ht!]
\begin{center}
\begin{tikzpicture}[scale=0.2]

\draw[ultra thick,black] (-3,0) parabola (3,-6)parabola (12,-8) parabola(17,4)parabola (12,12)  parabola (-3,0) -- (-10,0.5) -- (-9.5,-3.5) -- (-3,0);

\draw[ultra thick,red] (-3,0) parabola (17,4);
\draw[black] (-3,-0.1) node[below] {$\vecx$};
\draw[black] (-3,0) node[circle, draw, fill=black!50,inner sep=0pt, minimum width=4pt] {};
\draw[black] (17.1,4) node[right] {$\vecy$};
\draw[black] (17,4) node[circle, draw, fill=black!50,inner sep=0pt, minimum width=4pt] {};

\draw[scale=1,black] (9,1.2) node[below] {$c_{k+1}$};
\draw[black] (-11,-0.5) node[below] {$\sigma_{k}$};

\end{tikzpicture}
\caption{$\sigma_{k+1}=\sigma_{k} \cup c_{k+1}$. }
\end{center}
\end{figure}
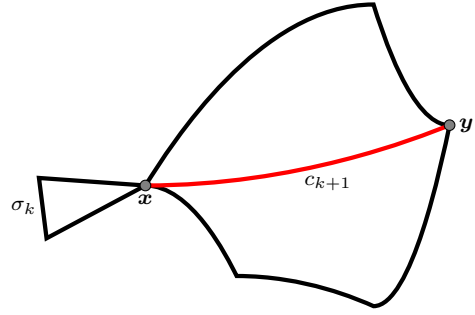
\

For $j = 1,\dots,n$, let $\ell_j = \ls[\vecz_{j-1},\vecz_j]$. Define subsets $I_1,I_2,I_3 \subseteq \{1,2,\dots,n\}$ by
   \begin{align*}
    I_1 &= \{ j \st \vecz_j,\vecz_{j-1} \in \sigma_{k}\}, \\
    I_2 &= \{ j \st \vecz_j,\vecz_{j-1} \in c_{k+1}\}, \\
    I_3 &= \{1,2,\dots,n\} \setminus (I_1 \cup I_2).
  \end{align*}

\begin{figure}[ht!]
\begin{center}
\begin{tikzpicture}[scale =0.2]
\draw[ultra thick,black] (-3,0) parabola (3,-6)parabola (12,-8) parabola(17,4)parabola (12,12)  parabola (-3,0) -- (-10,0.5) -- (-9.5,-3.5) -- (-3,0);

\draw[ultra thick,red] (-3,0) parabola (17,4);
\draw[black] (-3,0) node[below] {$\vecx$};
\draw[black] (-3,0) node[circle, draw, fill=black!50,inner sep=0pt, minimum width=4pt] {};
\draw[ultra thick,blue] (-6,0.4)--(0,4.2)--(2,0.3)--(1,5.8)--(13,9)--(14,-6)--(3,-6) --(5,0.5)--(9,1.3)--(5,9.5)--(17,4);

\draw[thick, black] (-11,-0.5) node[below] {$\sigma_{k}$};

\draw[black] (-6,0.4) node[above] {$\vecz_0$};
\draw[black] (-6,0.4) node[circle, draw, fill=black!50,inner sep=0pt, minimum width=4pt] {};
\draw[black] (0,4.2) node[left] {$\vecz_{1}$};
\draw[black] (0,4.2) node[circle, draw, fill=black!50,inner sep=0pt, minimum width=4pt] {};
\draw[black] (2,0.3) node[below] {$\vecz_2$};
\draw[black] (2,0.3) node[circle, draw, fill=black!50,inner sep=0pt, minimum width=4pt] {};
\draw[black] (1,5.8) node[above] {$\vecz_3$};
\draw[black] (1,5.8) node[circle, draw, fill=black!50,inner sep=0pt, minimum width=4pt] {};
\draw[black] (13,9) node[right] {$\vecz_4$};
\draw[black] (13,9) node[circle, draw, fill=black!50,inner sep=0pt, minimum width=4pt] {};
\draw[black] (5,0.5) node[above] {$\vecz_7$};
\draw[black] (5,0.5) node[circle, draw, fill=black!50,inner sep=0pt, minimum width=4pt] {};
\draw[black] (9,1.3) node[below] {$\vecz_8 $};
\draw[black] (9,1.3) node[circle, draw, fill=black!50,inner sep=0pt, minimum width=4pt] {};
\draw[black] (17,4) node[right] {$\vecz_{10}=\vecy$};
\draw[black] (17,4) node[circle, draw, fill=black!50,inner sep=0pt, minimum width=4pt] {};
\draw[black] (14,-6) node[right] {$\vecz_5$};
\draw[black] (14,-6) node[circle, draw, fill=black!50,inner sep=0pt, minimum width=4pt] {};
\draw[black] (3,-6) node[below] {$\vecz_6$};
\draw[black] (3,-6) node[circle, draw, fill=black!50,inner sep=0pt, minimum width=4pt] {};
\draw[black] (5,9.5) node[above] {$\vecz_{9}$};
\draw[black] (5,9.5) node[circle, draw, fill=black!50,inner sep=0pt, minimum width=4pt] {};

\end{tikzpicture}
\caption{In this example $I_1 = \{1,4,5,6,10\}$, $I_2 = \{8\}$ and $I_3 = \{2,3,7,9\}$. The sublist $S_1$ is $[\vecz_0,\vecz_1,\vecz_3,\vecz_4,\vecz_5,\vecz_6,\vecz_9,\vecz_{10}]$ and the sublist $S_2$ is $[\vecz_7,\vecz_8]$. }
\end{center}
\end{figure}
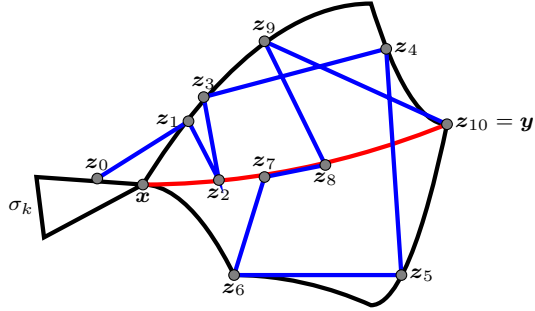

Note that if $j \in I_3$ then one end of $\ell_j$ must line in $\sigma_k \setminus c_{k+1}$ and the other must lie in $c_{k+1} \setminus \sigma_k$.

Noting that $I_1 \cap I_2$ may be nonempty,
  \begin{equation*}\label{eq1}
    \cvar(f,S)
    \le \sum_{i = 1}^3 \sum_{j \in I_i} |f(\vecz_j) - f(\vecz_{j-1})|.
  \end{equation*}
Form the sublist $S_1$ of $S$ by including all points which are endpoints of line segments $\ell_j$ with $j \in I_1$ (and dropping points if necessary to ensure that consecutive points are distinct). By [DL2,Proposition 3.5] and \cite{AD2}, $\vf(S_1) \leq \vf(S)$, and so we have
  \begin{equation}\label{eq2}
    \frac{\sum_{j \in I_1} |f(\vecz_j) - f(\vecz_{j-1})|}{\vf(S)}
     \le \frac{\cvar(f,S_1)}{\vf(S_1)}
     \le \var(f,\sigma_{k}).
  \end{equation}
Similarly, if $S_2$ is the sublist of $S$ including all points which are endpoints of line segments $\ell_j$ with $j \in I_2$ then
  \begin{equation}\label{eq3}
    \frac{\sum_{j \in I_2} |f(\vecz_j) - f(\vecz_{j-1})|}{\vf(S)}
      \le \frac{\cvar(f,S_2)}{\vf(S_2)}
      \le \var(f,c_{k+1}).
  \end{equation}

Consider now the polygon $P_{k+1}$ which contains $c_{k+1}$. If $j \in I_3$, then $\ell_j$ is a crossing segment on at least one of the lines which form the boundary of $P_{k+1}$.  In particular, at least one of the $S(P_{k+1})$ lines must have at least $|I_3|/S(P_{k+1})$
crossing segments, and hence $\vf(S) \ge |I_3|/S(P_{k+1})$. By a simple triangle inequality estimate
\begin{equation}\label{eq4}
\frac{\sum_{j \in I_3} |f(\vecz_{j} - f(\vecz_{j-1})|}{\vf(S)}
    \le \frac{2 |I_3| \norm{f}_{\infty,\sigma_{k+1}} }{|I_3|/S(P_{k+1})}
   = 2 S(P_{k+1}) \norm{f}_{\infty,\sigma_{k+1}}
\end{equation}

Combining the three estimates (\ref{eq2}), (\ref{eq3}) and (\ref{eq4}) we see that
\[\var(f,\sigma_{k+1}) \le \var(f,\sigma_{k}) + \var(f,c_{k+1}) + 2 S(\calM) \norm{f}_{\infty,\sigma_{k+1}}\]
and hence
\begin{align*}
\norm{f}_{BV(\sigma_{k+1})}
  & = \norm{f}_{\infty,\sigma_{k+1}} +  \var(f,\sigma_{k+1}) \\
  &\leq \norm{f}_{\infty,\sigma_{k}} + \norm{f}_{\infty,c_{k+1}}
     + \var(f,\sigma_{k}) + \var(f,c_{k+1})
     + 2 S(\calM)   \norm{f}_{\infty,\sigma_{k+1}} \\
  & \leq (\bigl(k+2(k-1)S(\calM)\bigr) \Bigl(\norm{f}_{\infty,\sigma_{k}}
         + \sum_{j=1}^{n} \var(f,c_{j}) \Bigr)\\
  & \qquad
     + (1+2 S(\calM)) \norm{f}_{\infty,\sigma_{k+1}} + \var(f,c_{k+1}) \\
  &\leq \bigl(k+2(k-1)S(\calM)+1+2 S(\calM)\bigr) \norm{f}_{PIC(\sigma_{k+1})} \\
 &= \bigl(k+1+2kS(\calM)\bigr) \norm{f}_{PIC(\sigma_{k+1})}
\end{align*}
and so (\ref{bv-pic-induction}) holds when $m=k+1$.
\end{proof}

The constant $K_\sigma = M+2(M-1)S(\calM)$ obtained in the above proof depends on the particular polygonal mosaic used to decompose $\sigma$. The exact way in which the best constant depends on $\sigma$ is not known, but examples such as the one in Figure~\ref{bad-bv-ex} show that there is no bound which is independent of the number of component curves.

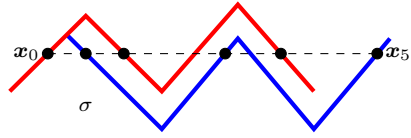
\begin{figure}
\begin{center}
\begin{tikzpicture}

\draw[blue,ultra thick] (0.75,0.75) -- (2,-0.5) -- (3,0.7) -- (4,-0.5) -- (5,0.7);
\draw[red,ultra thick] (0,0) -- (1,1) -- (2,0) -- (3,1.15) -- (4,0);
\draw[black,dashed] (0.5,0.5) -- (4.8333,0.5);
\filldraw[black] (0.5,0.5) circle (2pt);
\draw (0.5,0.5) node[left] {$\vecx_0$};
\filldraw[black] (1.0,0.5) circle (2pt);
\filldraw[black] (1.5,0.5) circle (2pt);
\filldraw[black] (2.8333,0.5) circle (2pt);
\filldraw[black] (3.566,0.5) circle (2pt);
\filldraw[black] (4.8333,0.5) circle (2pt) ;
\draw (4.8333,0.5) node[right] {$\vecx_5$};

\draw (1,-0.2) node {$\sigma$};

\end{tikzpicture}
\caption{Suppose that $f:\sigma \to \mC$ is the characteristic function of the upper (red) curve. Then $\norm{f}_{\PIC(\sigma)} = 2$, while by considering the list $S = [\vecx_0,\dots,\vecx_5]$, one can see that $\norm{f}_{BV(\sigma)} \ge 6$. By adding additional components to $\sigma$ one can clearly make the $BV$ norm as large as one likes while not increasing the $\PIC$ norm.}\label{bad-bv-ex}
\end{center}
\end{figure}

\begin{corollary}\label{BV-PIC-equiv} Suppose that $\sigma \in \PIC$.
Then $\norm{\cdot}_{\PIC(\sigma)}$ is equivalent to $\norm{\cdot}_{BV(\sigma)}$.
\end{corollary}

\begin{corollary}\label{BV-isom-PIC}
Suppose that $\sigma,\tau \in \PIC$. If $\sigma$ is homeomorphic to $\tau$ then $BV(\sigma)$ is isomorphic to $BV(\tau)$.
\end{corollary}

\begin{proof}
Using the construction in Section~\ref{graph-isom}, we can write $\sigma = \cup_{i=1}^n c_i$ and $\tau = \cup_{i=1}^n c_i'$ with a homeomorphism $h:\sigma \to \tau$ which maps $c_i \to c_i'$ via the arc-length parameterization. For $f: \sigma \to \mathbb{C}$ define
  \[ \Phi(f)(\vecz) = f(h^{-1}(\vecz)), \qquad \vecz \in \tau. \]
Then $\pvar(f,c_i) = \pvar(\Phi(f),c_i')$ and so $\norm{f}_{\PIC(\sigma)} = \norm{\Phi(f)}_{\PIC(\tau)}$. By Corollary~\ref{BV-PIC-equiv} this means that $\Phi$ is a continuous map from $BV(\sigma)$ to $BV(\tau)$. Since $\Phi^{-1}(g) = g \circ h$, it follows that $\Phi$ is actually an isomorphism.
\end{proof}

\section{Absolutely continuous functions}\label{S:AC-functions}

The isomorphism $\Phi: BV(\sigma) \to BV(\tau)$ defined in the proof of Corollary~\ref{BV-isom-PIC} is of the form $\Phi(f) = f \circ h^{-1}$ for a particular homeomorphism $h: \sigma \to \tau$. In fact the choice of homeomorphism here is not particularly important. One just wants a function which maps each component curve $c_i \subseteq \sigma$ continuously onto the corresponding curve $c_i' \subseteq \tau$.
However, a badly chosen homeomorphism may not send $AC(\sigma)$ functions to $AC(\tau)$ functions. Our aim now is to show that one may choose a homeomorphism which preserves these subalgebras.

Let $c$ be the graph of a differentiable convex function $k: [0,1] \to \mC$. For $g: [0,1] \to \mC$, define $\Psi(g): c \to \mC$ by $\Psi(g)(x,k(x)) = g(x)$. Then, by Theorem~\ref{equiv-C},
  \begin{align*}
  \norm{g}_{BV[0,1]}
    & = \norm{g}_\infty+ \var(g,[0,1])\\
    &= \norm{\Psi(g)}_\infty+ \pvar(\Psi(g),c) \\
    &   \le 2\left(\norm{\Psi(g)}_\infty + \var(\Psi(g),c) \right) \\
    &  = 2 \norm{\Psi(g)}_{BV(c)} \\
    &  \le 2 \left(\norm{g}_\infty + \pvar(\Psi(g),c)\right) \\
    &  = 2 \norm{g}_{BV[0,1]}.
  \end{align*}
So $BV(c)$ is isomorphic to $BV[0,1]$. The more delicate thing is to check that $\Psi$ preserves absolute continuity.

\begin{proposition}\label{real-proj}
With $\Phi$ as above, $f \in AC[0,1]$ if and only if $\Psi(f) \in AC(c)$.
\end{proposition}

\begin{proof}
It is known (see \cite[Proposition~4.4]{AD}) that $\Psi$ is a norm-decreasing algebra homomorphism from $AC[0,1]$ to $AC(c)$.

Suppose first that $p \in AC(c)$ is a polynomial in two variables. Then $p_r(x) = p(x,k(x))$, is differentiable on $(0,1)$ with
  \[ p_r'(x) = \nabla p(x,k(x)) \cdot (1,k'(x)). \]
By the Fundamental Theorem for Line Integrals
  \[ p(x,k(x)) - p(0,k(0)) = \int_0^x \nabla p(s,k(s)) \cdot (1,k'(s)) \, ds \]
or
  \[ p_r(x) = p_r(0) + \int_0^x p_r'(s) \, ds \]
and hence $p_r$ is absolutely continuous. Of course $p_r = \Psi^{-1}(p)$.

Suppose now that $g \in AC(c)$ and that $\epsilon > 0$. Then there exists a polynomial $p \in \calP_2$ such that $\norm{g - p}_{BV(c)} < \frac{\epsilon}{2}$. Then
  \begin{align*}
  \snorm{\Psi^{-1}(g) - \Psi^{-1}(p)}_{BV[0,1]}
    & = \snorm{\Psi^{-1}(g) - \Psi^{-1}(p)}_\infty + \var(\Psi^{-1}(g) - \Psi^{-1}(p),[0,1]) \\
    &= \norm{g-p}_\infty + \pvar(g - p,c) \\
    & \le 2(\norm{g-p}_\infty + \var(g-p,c) \\
    & = 2 \norm{g-p}_{BV(c)} < \epsilon.
  \end{align*}
Since $AC[0,1]$ is a closed subalgebra of $BV[0,1]$ it follows that $\Psi^{-1}(g) \in AC[0,1]$. Thus $\Psi$ is a continuous Banach algebra isomorphism.
\end{proof}

Of course the map $h:[0,1] \to c$, $h(x) = (x,k(x))$ is a homeomorphism, and it is an easy consequence of the proposition that $\Phi(f) = f \circ h^{-1}$ is a Banach algebra isomorphism from $AC[0,1]$ to $AC(c)$. By a suitable rotation and rescaling we can conclude the following.

\begin{theorem}\label{AC-project}
Suppose that $c$ is a projectable convex curve in $\mR^2$. Then $AC(c) \simeq AC[0,1]$.
\end{theorem}

\section{A cut-off function lemma}\label{S:COL}

Every closed half-plane in $\mR^2$ can be written as $H = \{ \vecx \st (\vecx - \vecu) \cdot \vecv \ge 0\}$ for some $\vecu, \vecv \in \mR^2$, with $\norm{\vecv} = 1$. For $\epsilon > 0$ let
  \[ g_\epsilon(t) = \begin{cases}
                      0,   & \text{if $t \le \frac{\epsilon}{2}$} , \\
                      \frac{2t - \epsilon}{\epsilon}, & \text{if $\frac{\epsilon}{2} < t < \epsilon$}, \\
                      1,   & \text{if $t \ge \epsilon$. }
                     \end{cases}\]
If, as usual, $\sigma$ is a nonempty compact subset of the plane and we define $h_{H,\epsilon}: \sigma \to \mC$ by
  \[ h_{H,\epsilon}(\vecx) = g_\epsilon((\vecx - \vecu) \cdot \vecv) \]
then $h_\epsilon \in AC(\sigma)$.

Suppose that $P$ is a closed convex polygon. Then $P$ can be written as the intersection of closed half-planes, $P = \bigcap_{i=1}^n H_i$. Given  $\epsilon > 0$ we can define corresponding `cut-off' functions $h_{i,\epsilon}$ as above corresponding to these half planes. For $\epsilon$ sufficiently small the function $h_\epsilon = \prod_{i=1}^n h_{i,\epsilon}$ is then an $AC(\sigma)$ function which is zero on an open neighbourhood of the complement of $P$ and which is 1 on a smaller convex polygon in the interior of $P$. 

Recall that $\var(fg,P) \le \norm{f}_{\infty} \var(g,P) + \norm{g}_\infty \var(f,P)$. Since for each $i$, $\norm{h_i}_\infty = 1$ and $\var(h_i,P) = 1$, a simple induction proof shows that $\var(h_\epsilon,P) \le n$.
Consequently, if $c$ is any convex curve in $P$, then
  \begin{equation}\label{pvar-h}
    \pvar(h_\epsilon,c) \le 2 \var(h_\epsilon,c) \le 2 \var(h_\epsilon,P) \le 2n.
  \end{equation}

\section{The $AC$ joining lemma}\label{S:AC-join}

\begin{theorem} Suppose that $\sigma = \cup_{k=1}^n c_k \in \PIC$ is represented as a union of projectable convex curves. Let $\sigma_0 = \cup_{k=1}^{n-1} c_k$ be connected and suppose that $f \in BV(\sigma)$. Then $f \in AC(\sigma)$ if and only if
$f|\sigma_0 \in AC(\sigma_0)$ and $f|c_n \in AC(c_n)$.
\end{theorem}

\begin{proof}
The forward implication follows from the general results about restricting $AC$ functions (see \cite[Lemma 4.5]{AD}).

For the reverse direction, by scaling, rotating and reflecting as necessary we can assume that the endpoints of $c = c_n$ are $0$ and $1$, and that $c$ lies in the closed upper half-plane. We shall assume first that both $0$ and $1$ lie in $\sigma_0$; the case where $c$ joins to $\sigma_0$ at just one endpoint is similar. Let $P$ denote the polygon containing $c$ from a suitable polygonal mosaic for $\sigma$, and let $m$ denote the number of sides of $P$.

By Proposition~\ref{real-proj} the function $f_r(\Re \vecz) = f(\vecz)$, $\vecz \in c$ lies in $AC[0,1]$. We can therefore define $f_c: \sigma \to \mC$ by
  \[ f_c(\vecz) = \begin{cases}
                   f(0),  & \text{if $\Re \vecz < 0$,}\\
                   f_r(\Re \vecz),  & \text{if $0 \le \Re \vecz \le 1$,} \\
                   f(1),  & \text{if $\Re \vecz > 1$.}
                  \end{cases} \]
By \cite[Proposition~4.4]{AD}, $f_c \in AC(\sigma)$. It follows then that $g = f-f_c$ is in $BV(\sigma)$. Our aim is to show that $g \in AC(\sigma)$ and hence that $f = g+f_c$ is in $AC(\sigma)$.

Since by hypothesis $f|\sigma_0 \in AC(\sigma_0)$ and, by restriction, $f_c|\sigma_0 \in AC(\sigma_0)$ we have that $g|\sigma_0 \in AC(\sigma_0)$. Clearly $g|c$ is identically zero.

It will suffice now to show that there exists $q \in AC(\sigma)$ arbitrarily close to $g$, since this will imply that $g \in AC(\sigma)$. Fix $\epsilon > 0$.
Using the equivalence of the norms, there exists $p \in \mathcal{P}_2$ such that $\norm{g-p}_{\PIC(\sigma_0)} < \epsilon$. Then
  \[ |p(0)| = |p(0) - g(0)| \le \norm{p-g}_\infty \le \norm{p-g}_{\PIC(\sigma_0)} < \epsilon. \]
Similarly $|p(1)| < \epsilon$. Let $p_r: [0,1] \to \mC$ be defined by $p_r(\Re \vecz) = f(\vecz)$, for $\vecz \in c$. Then $p_r \in AC[0,1]$ so there exists $\delta > 0$ such that $\var(p_r,[0,\delta]) < \epsilon$ and $\var(p_r,[1-\delta,1]) < \epsilon$. It follows that $|p_r(t)| < 2\epsilon$ for $t \in [0,\delta\ \cup [1-\delta,1]$.

Consider the curves
 \begin{align*}
  c_L        &= \{\vecz \in c \st 0 \le \Re \vecz \le \delta\}, \\
  c_\delta   &= \{\vecz \in c \st \delta \le \Re \vecz \le 1-\delta\}, \\
  c_R        &= \{\vecz \in c \st 1-\delta \le \Re \vecz \le 1\}.
 \end{align*}
Since $c_\delta$ is a compact set, there is a positive minimum  distance $\eta$ from this set to the boundary of $P$. By the results of the previous section there exists a cut-off function $h \in AC(\sigma)$ such that $h(\vecz)= 1$ for $\vecz \in c_\delta$, and $h(\vecz) =0$ for $\vecz \in \sigma_0$.
Let $q = p(1-h)$. Then certainly $q \in AC(\sigma)$. Then
\allowdisplaybreaks
  \begin{align*}
   \norm{g-q}_{BV(\sigma)} & \le K_\mathcal{P} \norm{g-q}_{\PIC(\sigma)} \\
      & \le K_\mathcal{P} \Bigl( \norm{g-q}_{\infty} + \sum_{k=1}^n \pvar(g-q, c_k) \Bigr) \\
      &= K_\mathcal{P} \left( \norm{g-q}_{c,\infty} + \pvar(g-q, c) \right) \\
      &= K_\mathcal{P} \left( \norm{q}_{c,\infty} + \pvar(p(1-h), c) \right) \\
      &\le K_\mathcal{P} \bigl( 2 \epsilon + \pvar(p(1-h),c_L)  \\
      & \qquad\qquad  + \pvar(p(1-h),c_\delta) + \pvar(p(1-h),c_R) \bigr)  \\
      &\le K_\mathcal{P} \bigl( 2 \epsilon + \norm{p}_{c_L,\infty} \pvar(1-h,c_L)
                              + \pvar(p,c_L) \norm{1-h}_{c_L,\infty} \\
     & \qquad\qquad  + \norm{p}_{c_R,\infty} \pvar(1-h,c_R)
                              + \pvar(p,c_R) \norm{1-h}_{c_R,\infty} \bigr)\\
     & \le K_\mathcal{P} \bigl( 2 \epsilon + \epsilon  \pvar(1-h,c_L) + \epsilon  + \epsilon  \pvar(1-h,c_R) + \epsilon \bigr)\\
     & \le K_\mathcal{P} \bigl( 4 + 4m \bigr) \epsilon
  \end{align*}
using (\ref{pvar-h}). Since this can be made arbitrarily small by a suitable choice of $\epsilon$, we are done.

The case where $c$ joins $\sigma_0$ at just a single point is similar.
\end{proof}

An easy induction proof then shows the following.

\begin{corollary}\label{AC-patch}
 Suppose that $\sigma = \cup_{k=1}^n c_k \in \PIC$ is represented as a union of projectable convex curves. Suppose that $f \in BV(\sigma)$. Then $f \in AC(\sigma)$ if and only if
 $f|c_i \in AC(c_i)$ for all $i$.
\end{corollary}

We can now give our Gelfand-Kolmogorov type theorem for $\PIC$ sets.

\begin{theorem}
Suppose that $\sigma, \tau \in \PIC$. Then $AC(\sigma)$ is isomorphic to $AC(\tau)$ if and only if $\sigma$ is homeomorphic to $\tau$.
\end{theorem}

\begin{proof}
As noted in the introduction, we only need to show the reverse implication. Suppose then that $\sigma, \tau \in \PIC$ and that
$\sigma$ is homeomorphic to $\tau$.

We saw in Section~\ref{graph-isom} that we can find simple polygonal mosaics
for $\sigma$ and $\tau$ which split these sets up as drawings of isomorphic graphs. In particular we can write
  \[ \sigma = \bigcup_{i=1}^n c_i, \qquad \tau = \bigcup_{i=1}^n c_i'\]
where $c_i$ and $c_i'$ are matching edges of the associated graphs. If necessary one could use Proposition~\ref{projectable} and the Partition Lemma to further refine this decomposition so that each curve is projectable, so we will assume that all the curves have this property.

By Theorem~\ref{AC-project}, for each $i$, $AC(c_i)$ and $AC(c_i')$ are both isomorphic to $AC[0,1]$ via the homeomorphisms which projects these curves onto the line segments joining their endpoints and then rescales the interval. This in turn generates a homeomorphism $h_i: c_i \to c_i'$ whose orientation can be chosen to be consistent with the graph isomorphism in the way it maps endpoints (that is, graph vertices) from one set to another (see Figure~\ref{h_i-pic}).

 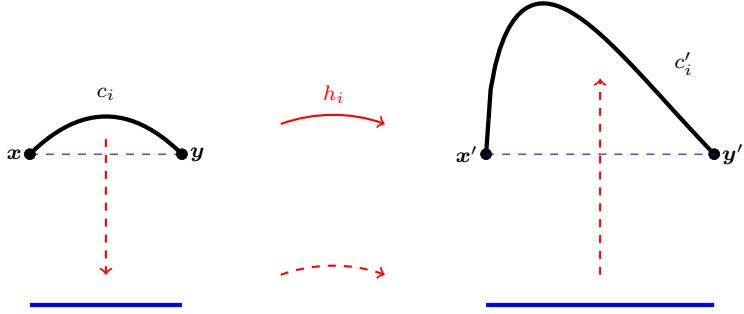
\begin{figure}[ht!]
 \begin{center}
  \begin{tikzpicture}[scale=2]
   \draw[ultra thick, black] (-2.0, 0.) -- (-1.96, 0.384e-1) -- (-1.92, 0.736e-1) -- (-1.88, .1056) -- (-1.84, .1344) -- (-1.80, .1600) -- (-1.76, .1824) -- (-1.72, .2016) -- (-1.68, .2176) -- (-1.64, .2304) -- (-1.60, .2400) -- (-1.56, .2464) -- (-1.52, .2496) -- (-1.48, .2496) -- (-1.44, .2464) -- (-1.40, .2400) -- (-1.36, .2304) -- (-1.32, .2176) -- (-1.28, .2016) -- (-1.24, .1824) -- (-1.20, .1600) -- (-1.16, .1344) -- (-1.12, .1056) -- (-1.08, 0.736e-1) -- (-1.04, 0.384e-1) -- (-1.00, 0.);
   \draw[blue, dashed] (-2,0) -- (-1,0);

   \draw (-1.5,0.4) node {$c_i$};
   \filldraw[black] (-2,0) circle (1pt);
   \filldraw[black] (-1,0) circle (1pt);
   \draw (-2,0) node[left] {$\vecx$};
   \draw (-1,0) node[right] {$\vecy$};

   \draw[blue, ultra thick] (-2,-1) -- (-1,-1);

   \draw[ultra thick, black] (1., 0.) -- (1.03, .4298) -- (1.06, .5878) -- (1.09, .6958) -- (1.12, .7762) -- (1.15, .8379) -- (1.18, .8858) -- (1.21, .9229) -- (1.24, .9511) -- (1.27, .9718) -- (1.30, .9863) -- (1.33, .9953) -- (1.36, .9995) -- (1.39, .9995) -- (1.42, .9958) -- (1.45, .9888) -- (1.48, .9788) -- (1.51, .9661) -- (1.54, .9511) -- (1.57, .9338) -- (1.60, .9147) -- (1.63, .8937) -- (1.66, .8712) -- (1.69, .8473) -- (1.72, .8221) -- (1.75, .7957) -- (1.78, .7683) -- (1.81, .7400) -- (1.84, .7108) -- (1.87, .6809) -- (1.90, .6504) -- (1.93, .6193) -- (1.96, .5878) -- (1.99, .5558) -- (2.02, .5235) -- (2.05, .4909) -- (2.08, .4581) -- (2.11, .4251) -- (2.14, .3920) -- (2.17, .3588) -- (2.20, .3256) -- (2.23, .2924) -- (2.26, .2593) -- (2.29, .2262) -- (2.32, .1933) -- (2.35, .1605) -- (2.38, .1279) -- (2.41, 0.9556e-1) -- (2.44, 0.6343e-1) -- (2.47, 0.3157e-1) -- (2.50, 0);

   \draw (2.3,0.6) node {$c_i'$};
   \filldraw[black] (1,0) circle (1pt);
   \filldraw[black] (2.5,0) circle (1pt);
   \draw (1,0) node[left] {$\vecx'$};
   \draw (2.5,0) node[right] {$\vecy'$};

   \draw[blue, dashed] (1,0) -- (2.5,0);

   \draw[blue, ultra thick] (1,-1) -- (2.5,-1);

   \draw[red,thick,->] (-0.35,0.2) arc (110:70:1);
   \draw[red] (0,0.4) node {$h_i$};

   \draw[red,thick, dashed,->] (-1.5, 0.1) -- (-1.5, -0.8);
     \draw[red,thick, dashed,->] (-0.35,-0.8) arc (110:70:1);
   \draw[red,thick, dashed,->] (1.75, -0.8) -- (1.75, 0.5);

  \end{tikzpicture}
 \caption{Each homeomorphism $h_i$ is a composition of projections and rescaling. One can choose the orientation in the middle step to make sure that if $\vecx'$ is the point in $\tau$ which corresponds under the graph isomorphism to $\vecx \in \sigma$ then $h_i$ maps $\vecx$ to $\vecx'$.}\label{h_i-pic}
 \end{center}
 \end{figure}

It follows that the map $h: \sigma \to \tau$ determined by $h|c_i = h_i$ is well-defined, and is a homeomorphism from $\sigma$ to $\tau$.

For $f: \sigma \to \mC$, let $\Phi(f) = f \circ h^{-1}$ be the corresponding function defined on $\tau$. As we have seen, $\Phi$ is a Banach algebra isomorphism from $BV(\sigma)$ to $BV(\tau)$.
By Corollary~\ref{AC-patch},
  \[ f \in AC(\sigma) \iff \text{$f|c_i \in AC(c_i)$ for all $i$}
   \iff \text{$\Phi(f)|c_i' \in AC(c_i')$ for all $i$}
   \iff \Phi(f) \in AC(\tau).
  \]
Thus, $AC(\sigma)$ is isomorphic to $AC(\tau)$.
\end{proof}

\section{A final remark}\label{S:Final-rem}

Most of the results in this article concern families of connected compact subsets of the plane. In this section we discuss how to deal with finite disjoint unions of such sets. The main step is proving Proposition~\ref{disj-poly}. First we need an extension lemma which is essentially given in the unpublished note \cite{DL2} (see Lemma~5.2).

\begin{lemma}\label{AC-ext}
Suppose that $\sigma_1,\sigma_2 \subseteq \mC$ are nonempty compact sets with $\Re \vecx < 0$ for all $\vecx \in \sigma_1$ and $\Re \vecx > 0$ for all $\vecx \in \sigma_2$. Let $\sigma = \sigma_1 \cup \sigma_2$. Suppose that $g \in AC(\sigma_1)$ and that ${\hat g}: \sigma \to \mC$ is defined by
  \[ {\hat g}(\vecx) = \begin{cases}
                   g(\vecx), & \text{if $\vecx \in \sigma_1$}, \\
                   0,          & \text{if $\vecx \in \sigma_2$}.
                  \end{cases}. \]
Then ${\hat g} \in AC(\sigma)$.
\end{lemma}

It is worth remarking that the separation of the two components is vital here. On cannot in general extend an absolutely continuous function on a compact set $\sigma_1$ to an absolutely continuous function on a superset by setting it to be zero off $\sigma_1$.

\begin{proof}
Fix $\epsilon > 0$. Since $g \in AC(\sigma_1)$ there exists a polynomial $p \in \calP_2$ such that $\norm{g - p}_{BV(\sigma_1)} < \frac{\epsilon}{2}$. Let $\chi$ denote the characteristic function of the left half-plane, restricted to $\sigma$. Then $\chi \in AC(\sigma)$ and so ${\hat p} = \chi p \in AC(\sigma)$ too. Let $\delta = {\hat g} - {\hat p}$.

Let $S = [\vecx_0,\vecx_1,\dots,\vecx_n]$ be an ordered list of elements in $\sigma$. Partition the indices $1,\dots,n$ into disjoint sets
  \begin{align*}
   J_1 &= \{j \st \vecx_{j-1}, \vecx_{j} \in \sigma_1 \}, \\
   J_2 &= \{j \st \vecx_{j-1}, \vecx_{j} \in \sigma_2 \}, \\
   J_3 &= \{1,\dots,n\} \setminus (J_1 \cup J_2).
  \end{align*}
Note that $\vf(S)$ must be at least as large as the number of elements in $J_3$.
Then, noting that $\delta$ is identically zero on $\sigma_2$, and treating empty sums as zero,
  \begin{align*}
   \sum_{j=1}^n |\delta(\vecx_j) - \delta(\vecx_{j-1})|
   & = \sum_{i=1}^3 \sum_{j \in J_1} |\delta(\vecx_j) - \delta(\vecx_{j-1})| \\
   &\le     \sum_{j \in J_1} |\delta(\vecx_j) - \delta(\vecx_{j-1})| +  |J_3| \norm{\delta}_\infty.
  \end{align*}
Now
 \[ \frac{|J_3| \norm{\delta}_\infty}{\vf(S)} \le \norm{\delta}_\infty  = \norm{\delta}_{\sigma_1,\infty} . \] 
If $J_1 \ne \emptyset$,  let $S_1 = [\vecx_{j_0}, \dots, \vecx_{j_\ell}]$ be the sublist of $S$ containing all the $x_j$ such that $x_j \in \sigma_1$ and at least one of $x_{j-1}$ or $x_{j+1}$ also lie in $\sigma_1$.  Since omitting points from a list can only decrease the variation, $\vf(S_1) \le \vf(S)$. Thus
  \begin{align*}
     \frac{\sum_{j \in J_1} |\delta(\vecx_j)- \delta(\vecx_{j-1})| }{\vf(S)}
     & \le \frac{\sum_{i=1}^\ell |\delta(\vecx_{j_i}) -\delta(\vecx_{j_{i-1}})| }{\vf(S_1)} \\
     & \le \var(\delta,\sigma_1).
  \end{align*}
It follows that (whether $J_1 = \emptyset$ or not)
  \[ \frac{\cvar(\delta,S)}{\vf(S)} \le \var(\delta,\sigma_1) + \norm{\delta}_{\sigma_1,\infty} = \norm{\delta}_{BV(\sigma_1)} \]
and so $\norm{\delta}_{BV(\sigma)} \le 2 \norm{\delta}_{BV(\sigma_1)} < \epsilon$. Thus $g \in AC(\sigma)$.
\end{proof}

Recall that if $\calA$ and $\calB$ are Banach algebras then $\calA \oplus \calB$ is a Banach algebra under componentwise operations and the norm $\norm{(a,b)} = \max\{\norm{a}_\calA,\norm{b}_\calB\}$.

\begin{proposition}\label{disj-poly}
Suppose that $P$ and $Q$ are disjoint polygons and that $\sigma \subseteq P \cup Q \subseteq \mR^2$ is a compact set such that $\sigma_P = \sigma \cap P$ and $\sigma_Q = \sigma \cap Q$ are both nonempty (and necessarily compact). Then $AC(\sigma)$ is isomorphic to $AC(\sigma_P) \oplus AC(\sigma_Q)$.
\end{proposition}

\begin{proof}
For $f: \sigma \to \mC$, let $J(f) = (f|\sigma_P,f|\sigma_Q)$.  By the general restriction theorems, if $f \in AC(\sigma)$ then $J(f) \in AC(\sigma_P) \oplus AC(\sigma_Q)$. Indeed $J$ is a norm 1 Banach algebra homomorphism. To complete the proof we need to show that $J$ is onto and that it has a continuous inverse.

Suppose then that $(f_P,f_Q) \in AC(\sigma_P) \oplus AC(\sigma_Q)$. Define ${\hat f}_P: \sigma \to \mC$ by
  \[ {\hat f}_P(\vecx) = \begin{cases}
                   f_P(\vecx), & \text{if $\vecx \in \sigma_P$}, \\
                   0,          & \text{if $\vecx \in \sigma_Q$}.
                  \end{cases} \]

Let $S$ be a large polygon which includes both $P$ and $Q$ in its interior. Following the algorithm in Section~7 of \cite{DL}, there exists a finite sequence of locally piecewise affine maps whose composition $h$ is a  homeomorphism of the plane which maps $S$ to a triangle and $P$ and $Q$ to triangles with disjoint projections on the real axis. Let $\tau = h(\sigma)$, $\tau_P = h(\sigma_P)$ and $\tau_Q = h(\sigma_Q)$.

\begin{figure}[ht!]
\begin{center}
 \begin{tikzpicture}[scale=2]

 \draw[fill,blue!10] (-1,-1) -- (1,-1) -- (1,1) -- (-1,1) -- (-1,-1);
 \draw[thick,black] (-1,-1) -- (1,-1) -- (1,1) -- (-1,1) -- (-1,-1);

 \draw[fill,white] (-0.7,-0.3) -- (-0.4,0.5) -- (0.8,0.5) -- (0.7,-0.5) -- (0.2,0) -- (0.6,0.3) -- (-0.4,0.2) -- (-0.7,-0.3);
 \draw[black] (-0.7,-0.3) -- (-0.4,0.5) -- (0.8,0.5) -- (0.7,-0.5) -- (0.2,0) -- (0.6,0.3) -- (-0.4,0.2) -- (-0.7,-0.3);

 \draw[fill,white] (0.1,0.1) -- (-0.3,-0.1) -- (-0.4,-0.8) -- (0,-0.3)  -- (0.2,-0.6) -- (0.1,0.1);
 \draw[black] (0.1,0.1) -- (-0.3,-0.1) -- (-0.4,-0.8) -- (0,-0.3)  -- (0.2,-0.6) -- (0.1,0.1);

 \draw[thick, blue] (-0.55,-0.0) -- (-0.34,0.4) -- (0.7,0.45) -- (0.65,-0.3) -- (0.35,-0.05);
 \draw[thick, blue] (-0.3,-0.6) -- (-0.1,-0.1) -- (0.13,-0.4);

  \draw (-0.9,0.87) node {$S$};
 \draw (0.7,0.6) node {$P$};
 \draw (0,-0.6) node {$Q$};

 \draw[red,thick,->] (1.7,0.2) arc (110:70:1);
 \draw[red] (2.1,0.4) node {$h$};

  \draw[fill,blue!10] (3,-1) -- (5,0) -- (3,1) -- (3,-1);
  \draw[thick,black] (3,-1) -- (5,0) -- (3,1) -- (3,-1);

   \draw[fill,white] (3.2,-0.2) -- (3.5,0.2) -- (3.8,-0.2) -- (3.2,-0.2);
  \draw[black] (3.2,-0.2) -- (3.5,0.2) -- (3.8,-0.2) -- (3.2,-0.2);

     \draw[fill,white] (3.9,-0.2) -- (4.2,0.2) -- (4.5,-0.2) -- (3.9,-0.2);
  \draw[black]  (3.9,-0.2) -- (4.2,0.2) -- (4.5,-0.2) -- (3.9,-0.2);

  \draw[thick, blue] (3.35, -0.15) -- (3.4,0.01) -- (3.45,.05) -- (3.55,-0.1) -- (3.5,0.1) -- (3.6,-0.1);
  \draw[thick, blue] (4.2,0.1) -- (4,-0.15 ) -- (4.2,0.02) -- (4.35,-0.1);
 \draw(3.3,0.3) node {$h(P)$};
 \draw(3.95,0.3) node {$h(Q)$};

 \end{tikzpicture}

\caption{Transforming polygons to triangles via a sequence of locally piecewise affine maps.}\label{poly-to-triangle}
\end{center}
\end{figure}
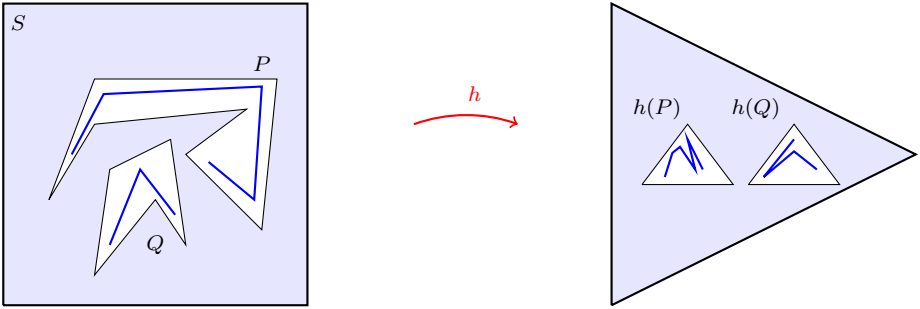

Since $h$ is a composition of locally piecewise affine maps, it generates an isomorphism $\Phi: AC(\sigma) \to AC(\tau)$ via $\Phi(f) = f \circ h^{-1}$.  This formula also determines an isomorphism $\Phi_P$ from $AC(\sigma_P)$ to $AC(\tau_P)$. In particular then $\Phi_P(f_p) \in AC(\tau_P)$.

By Lemma~\ref{AC-ext}, $\Phi_P(f_p)$ can be extended to a function $g \in AC(\tau)$ by setting $g|\tau_Q \equiv 0$. But ${\hat f}_P = \Phi^{-1}(g)$ and so ${\hat f}_P \in AC(\sigma)$.

Similarly, by setting it to be zero on $\sigma_P$, $f_Q$ can be extended to an absolutely continuous function ${\hat f}_Q$ on all of $\sigma$, and consequently $f ={\hat f}_P+{\hat f}_Q \in AC(\sigma)$. Clearly $J(f) = (f_P,f_Q)$ and hence $J$ is onto.

The continuity of $J^{-1}$ follows from the Banach Isomorphism Theorem.
\end{proof}

Let $\UPIC$ denote the family of compact subsets of the plane which are finite unions of pairwise disjoint sets $\sigma_m \in \PIC$, $m = 1,\dots,M$.

\begin{theorem}
Suppose that $\sigma, \tau \in \UPIC$. Then $AC(\sigma)$ is isomorphic to $AC(\tau)$ if and only if $\sigma$ is homeomorphic to $\tau$.
\end{theorem}

\begin{proof}
Suppose that $\sigma, \tau \in \UPIC$ are homeomorphic. Then they must have the same number of connected components, say $\sigma = \cup_{m=1}^M \sigma_m$ and $\tau = \cup_{m=1}^M \tau_m$. Furthermore, theses sets can be ordered so that, for each $m$,  $\sigma_m$ is homeomorphic to $\tau_m$,  and hence $AC(\sigma_m)$ is isomorphic to $AC(\tau_m)$.

Since these subsets are all compact, one can find disjoint polygons $P_1,\dots,P_M$ so that $\sigma_m$ lies in the interior of $P_m$. Hence by the last proposition
  \[ AC(\sigma) \simeq \bigoplus_{m=1}^M AC(\sigma_m) \simeq \bigoplus_{m=1}^M AC(\tau_m) \simeq AC(\tau).\]
\end{proof}

\begin{acknowledgements}
The work of the first author was financially supported
by the Ministry of Higher Education and Scientific Research
of Iraq.
\end{acknowledgements}

%
%

\end{document}